\documentclass[11pt,final,journal,onecolumn]{IEEEtran}
\usepackage{amsmath,amssymb,color,amsmath, graphicx, float, caption, subcaption, mathrsfs,color} \usepackage[ruled]{algorithm2e} 
\usepackage{bm}
\usepackage{epstopdf }
\usepackage{type1cm} 
\usepackage{lettrine}
\usepackage[T1]{fontenc}
\usepackage{selinput}
\SelectInputMappings{%
  eacute={é},
}
\usepackage[font={small}]{caption}
\usepackage[font={footnotesize}]{subcaption}
\usepackage{cite}
\usepackage{amsthm}
\usepackage{url}

\theoremstyle{plain}
\newtheorem{theorem}{\textbf{Theorem}}

\newtheorem{definition}[theorem]{\textbf{Definition}}
\newtheorem{lemma}[theorem]{\textbf{Lemma}}
\newtheorem{corollary}[theorem]{\textbf{Corollary}}

\makeatletter
\@addtoreset{theorem}{section}
\makeatother

\def\A{\bm{A}}

\def\D{\bm{D}}

\def\I{\bm{I}}

\def\L{\bm{L}}

\def\N{\bm{N}}

\def\P{\bm{P}}
\def\Q{\bm{Q}}
\def\R{\bm{R}}
\def\S{\bm{S}}

\def\U{\bm{U}}
\def\V{\bm{V}}

\def\X{\bm{X}}

\def\d{\bm{d}}
\def\e{\bm{e}}

\def\g{\bm{g}}

\def\i{\bm{i}}

\def\r{\bm{r}}

\def\u{\bm{u}}
\def\v{\bm{v}}

\def\x{\bm{x}}
\def\y{\bm{y}}
\def\z{\bm{z}}

\def\bSigma{\boldsymbol{\Sigma}}

\def\bOmega{\boldsymbol{\Omega}}

\def\balpha{\boldsymbol{\alpha}}

\def\bomega{\bm{\omega}}

\def\Ex{\operatorname{E}}

\def\CRLB_pe{\text{CRLB}}
\def\CRLB_pe{\text{CRLB}(\mathbf{p}_{e})}

\def\N02{\frac{N_0}{2}}

\def\0{\mathbf{0}}
\def\1{\mathbf{1}}

\SetKwInput{kwEvaluate}{Evaluate}
\SetKwInput{kwSort}{Sort}
\SetKwInput{kwInput}{Input}
\SetKwInput{kwOutput}{Output}
\SetKwInput{kwInitialize}{Initialize}



\begin{document}

%
%
%

\title{
Generalizing CoSaMP to Signals from a Union of Low Dimensional Linear Subspaces \\ }
\author{
    Tom~Tirer,
    and Raja~Giryes \\

    \thanks{This research was supported by a Grant from the GIF, the German-Israeli Foundation for Scientific Research and Development.

The authors are with the School of Electrical Engineering, Tel Aviv University, Tel Aviv 69978, Israel. (email: tirer.tom@gmail.com, raja@tauex.tau.ac.il)}
          }
\maketitle

\begin{abstract}

The idea that signals reside in a union of low dimensional subspaces subsumes many low dimensional models that have been used extensively in the recent decade in many fields and applications. Until recently, the vast majority of works have studied each one of these models on its own. However, a recent approach suggests providing general theory for low dimensional models using their Gaussian mean width, which serves as a measure for the intrinsic low dimensionality of the data. In this work we use this novel approach to study a generalized version of the popular compressive sampling matching pursuit (CoSaMP) algorithm, and to provide general recovery guarantees for signals from a union of low dimensional linear subspaces, under the assumption that the measurement matrix is Gaussian. We discuss the implications of our results for specific models, and use the generalized algorithm as an inspiration for a new greedy method for signal reconstruction in a combined sparse-synthesis and cosparse-analysis model. We perform experiments that demonstrate the usefulness of the proposed strategy.

\end{abstract}

\begin{IEEEkeywords}
Sparse representation, compressive sampling, CoSaMP, Gaussian mean width, union of subspaces
\end{IEEEkeywords}



\section{Introduction}
\label{sec1}

In the recent decade many signal and image processing applications have benefited remarkably from adopting low dimensional models for the signals of interest \cite{bruckstein2009sparse, nam2013cosparse, candes2009exact, recht2010guaranteed, baraniuk2009random, yu2012solving}. The most popular form of this low dimensional modeling assumes that signals admit a sparse representation in a given dictionary - this is referred to as the "sparsity model" \cite{bruckstein2009sparse}. Other options for low dimensional modeling include the analysis cosparse framework \cite{nam2013cosparse}, the low-rank model \cite{candes2009exact, recht2010guaranteed}, low dimensional manifolds \cite{baraniuk2009random} and Gaussian mixture models \cite{yu2012solving}. For each of them, vast amounts of algorithms and theory were developed \cite{elad2010book, eldar2012compressed, foucart2013mathematical}. The common factor of all these models is that they characterize data with a number of parameters, which is significantly smaller than the ambient dimension of the data.

Recently, it was suggested by several authors that the intrinsic low dimensionality of the data can be measured by its Gaussian mean width \cite{chandrasekaran2012convex, plan2013robust}, which is strongly related to the statistical dimension \cite{amelunxen2014living}. This new view enables analyzing low dimensional data that do not have an explicit representation and to unify existing results, providing general theory and algorithms for signals residing in low dimensional models.

In this work we consider the problem of recovering an unknown signal $\x \in \Bbb R^n$ from $m$ linear noisy measurements of the form
\begin{align}
\label{Eq_model}
\y = \A \x + \e,
\end{align}
where $\A$ is an $m \times  n$ matrix with entries independently drawn from the standard normal distribution and $\e \in \Bbb R^m$ represents the noise. We assume that the noise does not depend on $\A$, and that its energy is bounded. Further, we assume that $m < n$.
As in this case, recovering $\x$ from $\y$ becomes an ill-posed problem, an additional prior on $\x$ is required. For example, in compressive sampling \cite{candes2008introduction, duarte2008single} the signal is assumed to admit a sparse representation, which allows its stable recovery from the measurements. In this work we rely on a more general signal model, a union of low dimensional linear subspaces.
More specifically, we assume that a possibly infinite set of finite-dimensional subspaces is given, and that $\x$ resides in one of these subspaces. However, we do not know a priori in which one.
This model subsumes many of the priors that were mentioned above, such as sparse representation; and more.

In the context of compressive sampling, two major strategies are used to recover a sparse signal from incomplete and possibly inaccurate samples. The first approach uses convex relaxation to formulate a convex program whose minimizer is known to approximate the target signal \cite{chen2001atomic, candes2005decoding, candes2006stable, candes2007dantzig}. The second one uses greedy methods to iteratively build an approximation, making locally optimal choices at each iteration. Examples of greedy programs include matching pursuit (MP) \cite{mallat1993matching}, orthogonal matching pursuit (OMP) \cite{chen1989orthogonal, pati1993orthogonal}, regularized orthogonal matching pursuit (ROMP) \cite{needell2010signal}, iterative hard thresholding (IHT) \cite{blumensath2009iterative}, subspace pursuit (SP) \cite{dai2009subspace} and compressive sampling matching pursuit (CoSaMP) \cite{needell2009cosamp}.
Recovery guarantees for the methods above (of both strategies) are mostly proved using the assumption that the measurement matrix $\A$ has the restricted isometry property (RIP) \cite{candes2005decoding}, with a sufficiently small RIP-constant.
In general, convex relaxation methods have better recovery guarantees than greedy methods, but they are also more computationally demanding. However, IHT, SP, and in particular CoSaMP, were shown to have recovery guarantees that are competitive with those of the convex relaxation methods.

In this work we focus on the greedy strategy for a general union of low dimensional subspaces. 
We study a generalized version of CoSaMP, and provide recovery guarantees that depend on the Gaussian mean width of a union of subspaces, which is strongly related to the one in which $\x$ resides.
We note that extensions of CoSaMP to different models exist \cite{baraniuk2010model, lee2010admira, davenport2013signal, giryes2015greedy, giryes2014greedy}. However, most of these works study specific models, while here we consider 
a general framework that includes these models and more, as discussed in Section \ref{Sec4}.
Moreover, Our proof technique, and assumptions on $\A$ and $\e$, facilitate obtaining bounds on the reconstruction error that are more robust to noise than previous results, which are based on RIP and its generalizations. 

The setting where $\x$ lies in a union of linear subspaces was studied in several works. Conditions for unique and stable sampling were developed in \cite{lu2008theory} for unions of infinitely many finite-dimensional subspaces, and in \cite{blumensath2009sampling} for finite unions of finite-dimensional subspaces. In \cite{puy2015recipes} the sampling theory was extended to models beyond unions of subspaces. However, all these works do not provide concrete recovery algorithms. 
Two prominent works that use results from \cite{blumensath2009sampling} and propose recovery algorithms are \cite{eldar2009robust} and \cite{baraniuk2010model}. In \cite{eldar2009robust} the modeling of each subspace in the union is rather general (though, less general than in this paper), but the number of possible subspaces is assumed to be finite and known, and the proposed recovery algorithm, which is based on convex relaxation, deeply depends on this assumption. In \cite{baraniuk2010model}, similarly to the work we present here, a generalized version of CoSaMP was studied. However, the signal model in \cite{baraniuk2010model} is restricted to predefined sparsity patterns under a predefined basis, thus the number of possible subspaces is known and the guarantees do not cover other models (such as sparsity under a given redundant dictionary). The model that we consider here does not assume that the union is finite. Hence, it includes specific models that are not included in \cite{eldar2009robust, baraniuk2010model}. One such example is a signal that can be formed as a rank-$k$ matrix.
Perhaps the most related work to the current paper is \cite{blumensath2011sampling}, which also considers reconstruction of signals from a possibly infinite union of linear subspaces. However, \cite{blumensath2011sampling} study the projected Landweber algorithm, which is basically a generalization of IHT to models beyond sparsity. Here we build on CoSaMP, which is stronger than IHT in general. In addition, our analysis shows more robustness to noise than the analysis in \cite{blumensath2011sampling}, which is based on assumptions that can be seen as a generalization of RIP.

\textbf{Contribution.} In this work we make the following contributions. We introduce a generalized CoSaMP (GCoSaMP) algorithm that relies on a general signal model - union of low dimensional subspaces. We show that when the number of measurements $m$ is large enough with respect to some constant $m_0$, which scales as the Gaussian mean width of the union of subspaces, we have that the reconstruction result $\hat{\x}$ satisfies
\begin{align}
\label{Eq_main_cont}
\|\hat{\x} - \x \|_2 \leq c_0 + c_1\|\e\|_2
\end{align}
with high probability, where $c_0$ is a constant that tends to zero as the number of iterations grows (to be more precise $c_0=\mathcal{O}(m^{-t/2})$, where $t$ is the number of iterations), and $c_1$ is another constant that obeys $c_1=\mathcal{O}(m^{-1})$. The significance of this result, except for its generality, is in showing that the effect of any stationary noise, that does not depend on $\A$, tends to zero as the number of measurements grows. The recovery guarantees hold for any algorithm that is an exact instance of GCoSaMP. Two such examples are CoSaMP and ADMiRA \cite{lee2010admira}. Another main contribution of the paper is introducing a new practical method for signal reconstruction, and possible decomposition, in a combined sparse synthesis and cosparse analysis model. The proposed technique is a relaxed version of GCoSaMP, specialized to this combined model. The idea of combining two sparse models in order to describe a signal has already been practiced in \cite{elad2005simultaneous}, where two adapted dictionaries (e.g., one dictionary adapted to represent textures, and the other to represent cartoons) were used. However, our scheme is significantly different than the one in \cite{elad2005simultaneous}. Our solution covers also the case of compressive and limited measurements, uses the recently proposed cosparse analysis model instead of a second sparse synthesis model, and provides a natural way to combine the two models, and any other ones.

The remainder of the paper is organized as follows.
In Section \ref{Sec2} we present the notation we use in this work and some preliminaries.
In Section \ref{Sec3} we present the GCoSaMP algorithm and prove the main theorem of the paper. We also discuss the  significance of the results in showing the denoising capabilities of the algorithm.
In Section \ref{Sec4} we discuss the implications of our results for several specific models that fall into the general union of subspaces framework, and in particular, we present SACoSaMP, a new signal recovery algorithm for sparse synthesis and cosparse analysis combined model.
Experiments that support our theoretical results for GCoSaMP, and experiments that examine the performance of SACoSaMP, are given in Section \ref{Sec5}.
Section \ref{Sec6} concludes the paper.


\section{Notations and Preliminaries}
\label{Sec2}

In this section we present notations, definitions and basic results that will be used in our work. We write $\| \cdot \|_2$ for the Euclidean norm of a vector, and  $\| \cdot \|$ for the spectral norm of a matrix. 
We denote the unit Euclidean ball and sphere in $\Bbb R^n$ by $\Bbb B^n$ and $\Bbb S^{n-1}$, respectively, and the $n \times n$ identity matrix by $\I_n$.
By abuse of notation, $C,c$ denote positive absolute constants whose values may change from instance to instance.
Given a linear subspace $\mathcal{V} \subset \Bbb R^n$, we denote the orthogonal projection onto $\mathcal{V}$ by $\P_{\mathcal{V}}$, and the orthogonal projection onto the orthogonal complement of $\mathcal{V}$ by $\Q_{\mathcal{V}} = \I_n - \P_{\mathcal{V}}$. The sum of sets $\mathcal{K}_1$ and $\mathcal{K}_2$ is defined as
\begin{align}
\label{Eq_K1plusK2_def}
\mathcal{K}_1+\mathcal{K}_2 \triangleq \left \{ \x_1+\x_2 \, : \, \x_1 \in \mathcal K_1, \x_2 \in \mathcal K_2 \right \}.
\end{align}
We denote by $\mathcal S$ a known, possibly infinite, set of finite-dimensional linear subspaces $\{\mathcal{V}_i\}$.
The $B$-order sum for the set $\mathcal{S}$, with an integer $B \geq 1$, is defined as
\begin{align}
\label{Eq_S_B_def}
\mathcal{S}^B \triangleq \left \{ \sum \limits_{i=1}^{B} \mathcal{V}_i \, : \, \mathcal{V}_i \in \mathcal{S} \right \}.
\end{align}
Note that $\mathcal{S}^B$ is also a known, possibly infinite, set of linear subspaces, and $\mathcal{S} \equiv \mathcal{S}^1$.
The notation $\mathcal{U}^B$ stands for the union of all subspaces in the set $\mathcal{S}^B$, namely
\begin{align}
\label{Eq_U_def}
\mathcal{U}^B \triangleq \left \{ \bigcup _{i} \mathcal V_i^B \, : \, \mathcal V_i^B \in \mathcal{S}^B \right \}.
\end{align}
It is easy to show that $\mathcal{U}^B$ can also be expressed as
\begin{align}
\label{Eq_U_B_def2}
\mathcal{U}^B = \sum \limits_{i=1}^{B} \mathcal{U}^1 = \left \{ \sum \limits_{i=1}^{B} \x_i \, : \, \x_i \in \mathcal{U}^1 \right \}.
\end{align}
We assume that the signal belongs to one of the subspaces in $\mathcal{S}$, i.e., $\x \in \mathcal{V}_0$, and $\mathcal{V}_0 \in \mathcal{S}$. However, this subspace is unknown, and therefore our signal model is $\x \in \mathcal{U}$, where $\mathcal{U} \equiv \mathcal{U}^1$.

To exemplify the definitions above, we consider the traditional sparsity model $\{ \x \in \Bbb R^n : \|\x\|_0 \leq k \}$, where $\|\cdot\|_0$ denotes the $\ell_0$ pseudo-norm, that counts the number of nonzero elements of a vector. Clearly, if $\mathcal S$ is the set of $\binom{n}{k}$ subspaces, where each subspace is spanned by a different selection of $k$ columns of the identity matrix $\I_n$, then $\x$ resides in one of the subspaces in $\mathcal S$, and the sparsity model can be compactly expressed as $\x \in \mathcal{U}$. Moreover, in this case, the set $\mathcal S^B$ contains all the $\binom{n}{Bk}$ subspaces spanned by different selections of $Bk$ columns of $\I_n$, and other subspaces of lower dimensions, which are contained in those $\binom{n}{Bk}$ subspaces. Lastly, from either (\ref{Eq_U_def}) or (\ref{Eq_U_B_def2}), we have $\mathcal U^B = \{ \x \in \Bbb R^n : \|\x\|_0 \leq Bk \}$.

\begin{definition}
\label{definition1}
The Gaussian mean width of a set $\mathcal{K} \in \Bbb R^n$ is defined as
\begin{align}
w(\mathcal{K}) \triangleq \Ex_{\g} \left \{ \underset{\z \in \mathcal{K}}{\operatorname{sup}} \langle \g,\z \rangle \right \}, \nonumber
\end{align}
where the expectation is taken over $\g \sim \mathcal{N}(\0,\I_n)$, a vector of independent zero-mean unit-variance Gaussians.
\end{definition}

The following lemma results from \cite[Cor 1.2]{gordon1988milman} together with standard concentration of measure for Gaussian random variable \cite{ledoux2013probability}.

\begin{lemma}
\label{lemma_gordon}
Let $\mathcal{K} \in \Bbb R^n$ be a closed set and let $\A$ be an $m \times n$ matrix with entries independently drawn from the standard normal distribution $\mathcal{N}(0,1)$. Also define
\begin{align}
b_m = \Ex \left \{ \| \g \|_2 \right \}, \nonumber
\end{align}
where $\g \in \Bbb \R^m$ is distributed as $\mathcal{N}(\0,\I_m)$. Then for all $\u \in \mathcal{K}$ and $\eta>0$ we have 
\begin{align}
\label{Eq_gordon_lb}
\frac{\| \A \u \|_2}{\| \u \|_2} \geq b_m - w(\mathcal{K} \cap \Bbb S^{n-1}) - \eta 
\end{align}
holds with probability at least $1-\mathrm{e}^{-\frac{\eta^2}{2}}$. Furthermore, for all $\u \in \mathcal{K}$ and $\eta>0$ we have 
\begin{align}
\label{Eq_gordon_ub}
\frac{\| \A \u \|_2}{\| \u \|_2} \leq b_m + w(\mathcal{K} \cap \Bbb S^{n-1}) + \eta 
\end{align}
holds with probability at least $1-\mathrm{e}^{-\frac{\eta^2}{2}}$.
\end{lemma}

\begin{lemma}
\label{lemma_intro}
Consider the assumptions and notations of Lemma \ref{lemma_gordon}. If the positive parameters $\mu$ and $\eta$ satisfy $\mu\leq \left ( b_m+w(\mathcal{U}^B \cap \Bbb S^{n-1})+\eta \right )^{-2}$ then for all $\mathcal{V} \in \mathcal{S}^B$, we have
\begin{align}
\label{Eq_intro}
\| \P_{\mathcal{V}} (\I_n - \mu \A^*\A) \P_{\mathcal{V}} \| \leq 1- \mu \left [b_m - w \left (\mathcal{U}^B \cap \Bbb S^{n-1} \right )-\eta \right ]^2
\end{align}
holds with probability at least $1-2\mathrm{e}^{-\frac{\eta^2}{2}}$.
\end{lemma}

\begin{proof}
Note that if $\z \in \Bbb R^n$ has unit Euclidean norm, then $\P_{\mathcal{V}} \z \in \mathcal{V} \cap \Bbb B^{n}$. Since $\mathcal{V} \cap \Bbb B^{n} \subset \mathcal{U}^B \cap \Bbb B^n$, using Lemma \ref{lemma_gordon} we have
\begin{align}
\label{Eq_intro_gordon_lb}
\frac{\left \| \A \P_{\mathcal{V}} \z \right \|_2^2}{\left \| \P_{\mathcal{V}} \z \right \|_2^2} \geq \left [b_m - w(\mathcal{U}^B \cap \Bbb S^{n-1}) - \eta \right ]^2
\end{align}
and
\begin{align}
\label{Eq_intro_gordon_ub}
\frac{\left \| \A \P_{\mathcal{V}} \z \right \|_2^2}{\left \| \P_{\mathcal{V}} \z \right \|_2^2} \leq \left [b_m + w(\mathcal{U}^B \cap \Bbb S^{n-1}) + \eta \right ]^2
\end{align}
hold together with probability at least $1-2\mathrm{e}^{-\frac{\eta^2}{2}}$. Assuming that this event occurred, we have
\begin{align}
\label{Eq_intro_derivation}
1- \mu \left [b_m - w \left (\mathcal{U}^B \cap \Bbb S^{n-1} \right )-\eta \right ]^2 & \geq \underset{\z : \|\z\|_2=1}{\operatorname{sup}} \left (1 - \mu \frac{\left \| \A \P_{\mathcal{V}} \z \right \|_2^2}{\left \| \P_{\mathcal{V}} \z \right \|_2^2} \right ) \\ \nonumber
& = \underset{\z : \|\z\|_2=1}{\operatorname{sup}} \left \| \P_{\mathcal{V}} \z \right \|_2^{-2} \left ( \left \| \P_{\mathcal{V}} \z \right \|_2^2 - \mu \left \| \A \P_{\mathcal{V}} \z \right \|_2^2 \right ) \\ \nonumber
& \geq \underset{\z : \|\z\|_2=1}{\operatorname{sup}} \left ( \left \| \P_{\mathcal{V}} \z \right \|_2 - \mu \left \| \A \P_{\mathcal{V}} \z \right \|_2^2 \right ) \\ \nonumber
& = \underset{\z : \|\z\|_2=1}{\operatorname{sup}} \langle \z , \P_{\mathcal{V}}(\I_n - \mu \A^*\A)\P_{\mathcal{V}} \z \rangle \\ \nonumber
& = \| \P_{\mathcal{V}} (\I_n - \mu \A^*\A) \P_{\mathcal{V}} \|.
\end{align}
The first inequality uses (\ref{Eq_intro_gordon_lb}) and $\mu>0$. The second inequality uses $\left \| \P_{\mathcal{V}} \z \right \|_2 \leq \left \| \z \right \|_2 = 1$ and the fact that the expression in the brackets is nonnegative, due to (\ref{Eq_intro_gordon_ub}) and the upper bound on $\mu$. The last equality follows from the fact that $\P_{\mathcal{V}}(\I_n-\mu \A^*\A)\P_{\mathcal{V}}$ is Hermitian.
\end{proof}


\section{Algorithm and Recovery Guarantees}
\label{Sec3}

The generalized CoSaMP (GCoSaMP) algorithm, which is a generalization of the sparsity-based algorithm proposed in \cite{needell2009cosamp}, is presented in Algorithm 1. The main difference compared to CoSaMP is that support selections are replaced by subspace selections. For a fixed value of $B$, we define a subspace selection problem as finding a subspace $\hat{\mathcal{V}} \in \mathcal{S}^B$, such that the orthogonal projection onto $\hat{\mathcal{V}}$ is the closest (in Euclidean distance) to the signal proxy/estimation $\tilde{\x}$, i.e.,
\begin{align}
\label{Eq_subspace selection}
\hat{\mathcal{V}} = \underset{\mathcal{V} \in \mathcal{S}^B}{\operatorname{argmin}} \| \tilde{\x} - \P_\mathcal{V} \tilde{\x} \|_2.
\end{align}
Although this problem may be demanding or intractable in general, for many models there exist efficient methods to obtain exact solutions, and for many others, approximation methods may still produce good results. In Section \ref{Sec4} we discuss several specific models.

\begin{algorithm}
\caption{Generalized CoSaMP (GCoSaMP)}
\vspace{2mm}
\kwInput{$\A, \y,$ stopping criterion, set of subspaces $\mathcal S$, where $\y = \A\x+\e$, such that $\e$ is an additive noise and $\x$ is an unknown signal satisfying $\x \in \mathcal{U}$.}
\kwOutput{$\hat{\x} \in \mathcal{U}$ an estimate for $\x$.}
\kwInitialize{$\r=\y, \x^0=0, t=0, \mathcal{V}^0=\varnothing$}
\While{stopping criterion not met}{
    $t = t+1$\;
    $\tilde{\v} = \A^*\r$\;
    $\mathcal{V}_{\Delta}^t = \underset{\mathcal{V} \in \mathcal{S}^2}{\operatorname{argmin}} \| \tilde{\v} - \P_\mathcal{V} \tilde{\v} \|_2$\;
    $\tilde{\mathcal{V}}^t = \mathcal{V}^{t-1} + \mathcal{V}_{\Delta}^t$\;
    $\tilde{\x}^t = \underset{\z}{\operatorname{argmin}} \| \y - \A\z \|_2 \,\,\, \operatorname{s.t.} \,\,\, \z \in \tilde{\mathcal{V}}^t$\;
    $\mathcal{V}^t = \underset{\mathcal{V} \in \mathcal S}{\operatorname{argmin}} \| \tilde{\x}^t - \P_{\mathcal{V}} \tilde{\x}^t \|_2$\;
    $\x^t = \P_{\mathcal{V}^{t}} \tilde{\x}^t$\;
    $\r = \y - \A \x^t$\;
}
$\hat{\x} = \x^t$\;
\end{algorithm}

Our main results are given in the following theorem and corollary.

\begin{theorem}
\label{theorem1}
Let $\y=\A\x+\e$, where $\x \in \mathcal U$ and $\A$ is an $m \times  n$ matrix with entries independently drawn from the standard normal distribution, then with the notation of Algorithm 1, we have
\begin{align}
\label{Eq_theorem}
\| \x^t - \x \|_2 \leq \frac{4 \rho_{2}(\eta)}{\sqrt{1-\rho_{1}^2(\eta)}} \| \x^{t-1} - \x \|_2 + \left ( \frac{4 \xi_{2}(\eta)}{\sqrt{1-\rho_{1}^2(\eta)}} + \frac{2 \xi_{1}(\eta)}{1-\rho_{1}(\eta)} \right ) \|\e\|_2,
\end{align}
holds for all $t$, with probability at least $1-6\mathrm{e}^{-\frac{\eta^2}{2}}$, where
\begin{align}
\label{Eq_theorem_rho1}
\rho_{1}(\eta) & \triangleq 1- \mu_1  \left [b_m - w \left (\mathcal U^{4} \cap \Bbb S^{n-1} \right )-\eta \right ]^2 \\
\label{Eq_theorem_xi1}
\xi_{1}(\eta) & \triangleq \mu_1 \left [ w \left ( \mathcal U^{3} \cap \Bbb S^{n-1} \right )+\eta \right ] \\
\label{Eq_theorem_rho2}
\rho_{2}(\eta) & \triangleq 1- \mu_2  \left [b_m - w \left (\mathcal U^{4} \cap \Bbb S^{n-1} \right )-\eta \right ]^2 \\
\label{Eq_theorem_xi2}
\xi_{2}(\eta) & \triangleq \mu_2 \left [ w \left ( \mathcal U^{4} \cap \Bbb S^{n-1} \right )+\eta \right ]
\end{align}
and $\mu_1$, $\mu_2$ and $\eta$ are positive parameters satisfying
\begin{align}
\mu_1 & \leq \left ( b_m+w(\mathcal U^{4} \cap \Bbb S^{n-1})+\eta \right )^{-2} \nonumber \\
\mu_2 & \leq \left ( b_m+w(\mathcal U^{4} \cap \Bbb S^{n-1})+\eta \right )^{-2}. \nonumber
\end{align}
The iterates converge if $m > 14.5^2 (w(\mathcal U^{4} \cap \Bbb S^{n-1})+\eta)^2+1$.
\end{theorem}

\begin{proof}
We divide the proof into several lemmas.

\begin{lemma}
\label{lemma2}
With the notation of Algorithm 1, we have
\begin{align}
\label{Eq_lemma2}
\| \x^t - \x \|_2 \leq 2 \| \tilde{\x}^t - \x \|_2.
\end{align}
\end{lemma}

\begin{lemma}
\label{lemma1}
With the notation of Algorithm 1 and Theorem \ref{theorem1}, we have
\begin{align}
\label{Eq_lemma1}
\| \tilde{\x}^t - \x \|_2 \leq \frac{1}{\sqrt{1-\rho_{1}^2(\eta)}} \| \Q_{\tilde{\mathcal{V}}^t}(\tilde{\x}^t-\x) \|_2 + \frac{\xi_{1}(\eta)}{1-\rho_{1}(\eta)} \| \e \|_2
\end{align}
holds for all $t$, with probability at least $1-3\mathrm{e}^{-\frac{\eta^2}{2}}$.
\end{lemma}

\begin{lemma}
\label{lemma3}
With the notation of Algorithm 1 and Theorem \ref{theorem1}, we have
\begin{align}
\label{Eq_lemma3}
\| \Q_{\tilde{\mathcal{V}}^t}(\tilde{\x}^t-\x) \|_2 \leq 2 \rho_{2}(\eta) \| \x - \x^{t-1} \|_2 + 2 \xi_{2}(\eta) \| \e \|_2
\end{align}
holds for all $t$, with probability at least $1-3\mathrm{e}^{-\frac{\eta^2}{2}}$.
\end{lemma}

The proofs of Lemmas \ref{lemma2}, \ref{lemma1} and \ref{lemma3} appear in Appendices A, B and C, respectively.
Combining the above three lemmas, we immediately get (\ref{Eq_theorem}). 

Note that  $\mu_1$ and $\mu_2$ are two arbitrary parameters (with restricted permitted range) that can be used to optimize the bound in (\ref{Eq_theorem}). Increasing $\mu_1$ leads to lower value of $\rho_1(\eta)$ and higher value of $\xi_1(\eta)$, and the same goes for $\mu_2$. However, in order to ensure that the iterates converge, it is enough to consider only the first term in (\ref{Eq_theorem}) and require that $\frac{4 \rho_{2}(\eta)}{\sqrt{1-\rho_{1}^2(\eta)}} <1$, which is equivalent to  $16\rho_{2}^2(\eta)+\rho_{1}^2(\eta)<1$.

To simplify the notations we define $m_0 \triangleq (w(\mathcal U^{4} \cap \Bbb S^{n-1})+\eta)^2$. Following the calculation from \cite{chandrasekaran2012convex}, we have
\begin{align}
\label{Eq_theorem_bm}
b_m \geq \frac{m}{\sqrt{m+1}} = \sqrt{\frac{m}{1+\frac{1}{m}}} \geq \sqrt{\frac{ 14.5^2 m_0 +1}{1+\frac{1}{m}}} > \sqrt{\frac{ 14.5^2 m_0 + \frac{1}{m}14.5^2 m_0}{1+\frac{1}{m}}} =14.5 \sqrt{m_0},
\end{align}
where in the last two inequalities we used $m \geq 14.5^2 m_0 +1$.
Setting both $\mu_1$ and $\mu_2$ at their highest permitted value $\left ( b_m+ \sqrt{m_0} \right )^{-2}$, we have
\begin{align}
\label{Eq_theorem_convergence}
16\rho_{2}^2(\eta)+\rho_{1}^2(\eta) & = 17 \left ( 1 - \frac{(b_m-\sqrt{m_0})^2}{(b_m+\sqrt{m_0})^2} \right )^2  < 1,
\end{align}
where the last inequality follows from (\ref{Eq_theorem_bm}).
\end{proof}
This leads us to the following corollary.

\begin{corollary}
Assuming the conditions of Theorem \ref{theorem1}, define $m_0 \triangleq (w(\mathcal U^{4} \cap \Bbb S^{n-1})+\eta)^2$ and assume $m \geq 14.5^2 m_0 +1$, then with the notation of Algorithm 1, we have
\begin{align}
\label{Eq_corollary}
\| \x^t - \x \|_2 \leq \rho^t(m) \| \x^{0} - \x \|_2 + \frac{\xi(m)}{1-\rho(m)} \|\e\|_2,
\end{align}
holds with probability at least $1-6\mathrm{e}^{-\frac{\eta^2}{2}}$, where
\begin{align}
\label{Eq_corollary_rho}
\rho(m) & \triangleq \operatorname{min} \left \{ 1, 4\frac{(\sqrt{m}+\sqrt{m_0})^2-(\frac{m}{\sqrt{m+1}}-\sqrt{m_0})^2}{(\frac{m}{\sqrt{m+1}} - \sqrt{m_0})(\sqrt{m}+\sqrt{m_0})} \right \} \approx \operatorname{min} \left \{ 1, \frac{16}{1 - \frac{m_0}{m}}\sqrt{\frac{m_0}{m}} \right \}, \\
\label{Eq_corollary_xi}
\xi(m) & \triangleq \frac{2 \sqrt{m_0} \sqrt{\frac{m+1}{m}}}{\frac{m^2}{m+1} - m_0} \left ( 2 + \frac{\sqrt{m}+\sqrt{m_0}}{\frac{m}{\sqrt{m+1}}-\sqrt{m_0}} \right ) \approx \frac{2 \sqrt{m_0}}{m-m_0} \frac{3\sqrt{m}-\sqrt{m_0}}{\sqrt{m}-\sqrt{m_0}}.
\end{align}
\end{corollary}

\begin{proof}
From Theorem \ref{theorem1}, using recursion we get
\begin{align}
\label{Eq_cor_recur}
\| \x^t - \x \|_2 \, \leq \, \tilde{\rho}^t \| \x^{0} - \x \|_2 + \frac{1-\tilde{\rho}^t}{1-\tilde{\rho}} \tilde{\xi} \|\e\|_2 \, \leq \, \tilde{\rho}^t \| \x^{0} - \x \|_2 + \frac{\tilde{\xi}}{1-\tilde{\rho}} \|\e\|_2,
\end{align}
where
\begin{align}
\label{Eq_cor_rho_def}
\tilde{\rho} & \triangleq \frac{4 \rho_{2}}{\sqrt{1-\rho_{1}^2}}, \\
\label{Eq_cor_xi_def}
\tilde{\xi} & \triangleq \frac{4 \xi_{2}}{\sqrt{1-\rho_{1}^2}} + \frac{2 \xi_{1}}{1-\rho_{1}},
\end{align}
and for brevity we omitted the dependencies on $\eta$. Similarly to the second part of the proof of Theorem \ref{theorem1}, we set both $\mu_1$ and $\mu_2$ at their largest permitted value $\left ( b_m+ \sqrt{m_0} \right )^{-2}$. We turn to bounding the basic terms in (\ref{Eq_cor_rho_def}) and (\ref{Eq_cor_xi_def}). We repeatedly use the fact $\sqrt{m} \geq b_m \geq \frac{m}{\sqrt{m+1}} > \sqrt{m_0}$. First, note that
\begin{align}
\label{Eq_cor_rho1_2_lb}
\rho_1 = \rho_2 &= 1 - \left ( \frac{b_m-\sqrt{m_0}}{b_m+\sqrt{m_0}} \right )^2 \leq 1 - \left ( \frac{\frac{m}{\sqrt{m+1}}-\sqrt{m_0}}{\sqrt{m}+\sqrt{m_0}} \right )^2 \triangleq 1-\phi^2, \\
\xi_1 &= \frac{w ( \mathcal U^{3} \cap \Bbb S^{n-1}) +\eta}{(b_m+\sqrt{m_0})^2} \leq \frac{\sqrt{m_0}}{(b_m+\sqrt{m_0})^2} \leq \frac{\sqrt{m_0}}{(\frac{m}{\sqrt{m+1}}+\sqrt{m_0})^2}, \\
\xi_2 &= \frac{\sqrt{m_0}}{(b_m+\sqrt{m_0})^2} \leq \frac{\sqrt{m_0}}{(\frac{m}{\sqrt{m+1}}+\sqrt{m_0})^2}.
\end{align}
Therefore,
\begin{align}
1-\rho_1^2 & \geq 1-\rho_1 \geq \phi^2,  \\
\sqrt{1-\rho_1^2} & \geq \phi, 
\end{align}
and we get
\begin{align}
\label{Eq_cor_tilde_rho}
\tilde{\rho} & \leq \frac{4(1-\phi^2)}{\phi} = 4(\phi^{-1}-\phi) = 4\frac{(\sqrt{m}+\sqrt{m_0})^2-(\frac{m}{\sqrt{m+1}}-\sqrt{m_0})^2}{(\frac{m}{\sqrt{m+1}} - \sqrt{m_0})(\sqrt{m}+\sqrt{m_0})}, \\
\tilde{\xi} & \leq \frac{4 \sqrt{m_0}/\phi}{(\frac{m}{\sqrt{m+1}}+\sqrt{m_0})^2} + \frac{2 \sqrt{m_0}/\phi^2}{(\frac{m}{\sqrt{m+1}}+\sqrt{m_0})^2} \\ \nonumber
& = \frac{2 \sqrt{m_0}(\sqrt{m}+\sqrt{m_0})}{(\frac{m}{\sqrt{m+1}}+\sqrt{m_0})^2(\frac{m}{\sqrt{m+1}}-\sqrt{m_0})}\left ( 2 + \frac{\sqrt{m}+\sqrt{m_0}}{\frac{m}{\sqrt{m+1}}-\sqrt{m_0}} \right ) \\ \nonumber
& \leq \frac{2 \sqrt{m_0} \sqrt{\frac{m+1}{m}} (\frac{m}{\sqrt{m+1}}+ \sqrt{m_0})}{(\frac{m}{\sqrt{m+1}}+\sqrt{m_0})^2(\frac{m}{\sqrt{m+1}}-\sqrt{m_0})}\left ( 2 + \frac{\sqrt{m}+\sqrt{m_0}}{\frac{m}{\sqrt{m+1}}-\sqrt{m_0}} \right ) \\ \nonumber
& = \xi(m).
\end{align}
Using (\ref{Eq_cor_recur}) and the last two inequalities we get (\ref{Eq_corollary}).

\end{proof}

For $m \gg m_0$, it is easy to see that $\rho(m)$ and $\xi(m)$ behave like $16\sqrt{\frac{m_0}{m}}$ and $6\frac{\sqrt{m_0}}{m}$, respectively. Therefore, the noise coefficient in (\ref{Eq_corollary}) also behaves like $6\frac{\sqrt{m_0}}{m}$. Assuming that the noise characteristics are the same at all measurements, it follows that $\|\e\|_2=\mathcal{O}(\sqrt{m})$. Thus, we conclude that the effect of the noise tends to zero as the number of measurements grows. However, please note that since our derivation assumes that $\e$ does not depend on $\A$ (see the remarks before (\ref{Eq_lemma1_e_term}) and (\ref{Eq_lemma3_main2_2})), the bound does not cover the case of adversarial noise $\e=-c\A\x$.

\textbf{Discussion.} Let us compare our bound with previous results. As was mentioned in Section \ref{Sec2}, by choosing $\mathcal S$ as the set of $\binom{n}{k}$ subspaces, where each subspace is spanned by a different selection of $k$ columns of the identity matrix $\I_n$, we get $\mathcal U = \{ \x \in \Bbb R^n : \|\x\|_0 \leq k \}$. That is to say, $\mathcal U \equiv  \mathcal{U}_{k\text{-sparse}}$ is the set of $k$-sparse vectors in $\Bbb R^n$. Also, the set $\mathcal S^B$ contains all the $\binom{n}{Bk}$ subspaces spanned by different selections of $Bk$ columns of $\I_n$ (and other subspaces of lower dimensions, which are contained in those $\binom{n}{Bk}$ subspaces). Therefore, in this case GCoSaMP reduces to the original sparsity-based CoSaMP \cite{needell2009cosamp}.
In order to use the theoretical results above, the evaluation of the quantity $w(\mathcal U^4_{k\text{-sparse}} \cap \Bbb S^{n-1})$ is required. Conveniently, since $\mathcal U^B_{k\text{-sparse}} = \{ \x \in \Bbb R^n : \|\x\|_0 \leq Bk \}$ we can use the following result from \cite{plan2013robust}
\begin{align}
\label{Eq_w_sparse_case}
cBk\,\mathrm{log}(2n/Bk) \leq w^2(\mathcal{U}^B_{k\text{-sparse}} \cap \Bbb S^{n-1}) \leq CBk\,\mathrm{log}(2n/Bk),
\end{align}
where $c$ and $C$ are two positive constants. Before we present the reconstruction error bound that was obtained in \cite{needell2009cosamp}, recall that we say that $\tilde{\A}$ satisfies the RIP of order $k$ with constant $\delta_k \in [0,1)$ if
\begin{align}
\label{Eq_rip_def}
(1-\delta_k)\|\x\|_2^2 \leq \|\tilde{\A}\x\|_2^2 \leq (1+\delta_k)\|\x\|_2^2
\end{align}
holds for all $\x$ satisfying $\|\x\|_0 \leq k$. In \cite[Thm 4.1]{needell2009cosamp} it was shown that given the measurements $\y=\tilde{\A}\x+\tilde{\e}$, where $\x$ is a $k$-sparse vector and $\tilde{\A}$ satisfies the RIP of order $4k$ with $\delta_{4k} \leq 0.1$, the solution of the $t$-th iteration $\x^t$ satisfies
\begin{align}
\label{Eq_cosamp_bound}
\| \x^t - \x \|_2 \leq \rho_{[25]}^t(m) \| \x^{0} - \x \|_2 + \frac{\xi_{[25]}(m)}{1-\rho_{[25]}(m)} \|\tilde{\e}\|_2,
\end{align}
where
\begin{align}
\label{Eq_cosamp_bound_rho_def}
\rho_{[25]} & \triangleq 2 \left ( 1+\frac{\delta_{4k}}{1-\delta_{3k}} \right ) \frac{\delta_{2k}+\delta_{4k}}{1-\delta_{2k}}, \\
\label{Eq_cosamp_bound_xi_def}
\xi_{[25]} & \triangleq 2 \left [ \frac{1}{\sqrt{1-\delta_{3k}}} + \left ( 1+\frac{\delta_{4k}}{1-\delta_{3k}} \right ) \frac{2\sqrt{1+\delta_{2k}}}{1-\delta_{2k}} \right ].
\end{align}
Note that we present the bound without substituting $\delta_{2k} \leq \delta_{3k} \leq \delta_{4k} \leq 0.1$, as done in \cite{needell2009cosamp}. Note also that the columns of $\tilde{\A}$ are assumed to be normalized in order to have $\delta_1=0$, while we did not normalize the columns of $\A$ to have unit variance (by dividing it by $\sqrt{m}$). Therefore, in order to have the same averaged noise level as in our model, we have $\|\tilde{\e}\|_2=\|\e\|_2/\sqrt{m}$, and if we also assume that the noise characteristics are the same at all measurements, it follows that $\|\tilde{\e}\|_2=\mathcal{O}(1)$. Now, it is easy to see that for fixed $k,n$, even if $m$ grows without bound such that the RIP constants tend to zero \cite{blumensath2009sampling}, the effect of the noise does not vanish in (\ref{Eq_cosamp_bound}), contrary to our result in (\ref{Eq_corollary}). The last remark is not surprising, since (\ref{Eq_cosamp_bound}) covers the case of adversarial noise. However, it emphasizes that our bound (\ref{Eq_corollary}) is more useful if the denoising capabilities of the algorithm are of interest (and not only its stability).

We turn to comparing (\ref{Eq_corollary}) with the bound obtained in \cite{giryes2012rip} for CoSaMP under the assumption of Gaussian noise. In \cite[Thm 2.3]{giryes2012rip} it was shown that given the measurements $\y=\tilde{\A}\x+\tilde{\e}$, where $\x$ is a $k$-sparse vector, $\tilde{\A}$ satisfies the RIP of order $4k$ with $\delta_{4k} \leq 0.1$ and $\tilde{\e} \sim \mathcal{N}(\0,\tilde{\sigma}^2 \I_m)$, after $t \geq t^*$ iterations we have that
\begin{align}
\label{Eq_cosamp_bound_gaussian}
\| \x^t - \x \|_2 \leq \tilde{C} \sqrt{2(1+a) k \, \mathrm{log}n} \cdot \tilde{\sigma}
\end{align}
holds with probability exceeding $1-(\sqrt{\pi(1+a)\mathrm{log}n}\cdot n^a)^{-1}$. The constant $\tilde{C}$ and the value of $t^*$ are given in \cite{giryes2012rip}.
For simplicity, we examine our bound in (\ref{Eq_corollary}) only for $m \gg m_0$. In this case, as we mentioned above, the noise coefficient behaves like $6\frac{\sqrt{m_0}}{m}$ and the first term can be as small as we desire for large enough $t$. Therefore, using (\ref{Eq_w_sparse_case}) and the definition of $m_0$, our bound in (\ref{Eq_corollary}) leads to
\begin{align}
\label{Eq_cosamp_bound_our}
\| \x^t - \x \|_2 \leq 6(\sqrt{4Ck\,\mathrm{log}(n/2k)}+\eta) \frac{\|\e\|_2}{m}
\end{align}
with probability exceeding $1-6\mathrm{e}^{-\frac{\eta^2}{2}}$. Once again, since we did not normalize the columns of $\A$, and assuming that the noise is stationary, we have that $\frac{\|\e\|_2}{m}$ is equivalent to $\tilde{\sigma}$.
Both bounds have quite similar dependencies on $k, n$, and they both agree that CoSaMP can recover signals with an effective reduction of the additive noise.
However, the bound in \cite{giryes2012rip} is more general in terms of the distribution of the measurement matrix, and the bound obtained here is more general in terms of the noise, which can be drawn from any distribution, or even exhibit nonstationarity. The only restriction on the noise is that it does not depend on $\A$. As a supporting example, in Section \ref{Sec5a} we conduct experiments with noise drawn from Laplace distribution.
We also stress that the bound obtained here can be applied to signal models beyond sparsity.
Lastly, we note that our assumptions on the measurement matrix and noise vector are similar to those in \cite{oymak2015sharp}, where bounds on the reconstruction error of projected gradient descent algorithm were developed.


\section{Examples}
\label{Sec4}

The results in Section \ref{Sec3} require the evaluation of the quantity $w(\mathcal{U}^4 \cap \Bbb S^{n-1})$, which depends on the subspaces included in the set $\mathcal S$ (recall the definitions in (\ref{Eq_S_B_def}) and (\ref{Eq_U_def})).
However, note that having an upper bound on $w(\mathcal U^4 \cap \Bbb S^{n-1})$ is enough for bounding the error in (\ref{Eq_corollary}). This is true since plugging an upper bound on $m_0$ in (\ref{Eq_corollary_rho}) and (\ref{Eq_corollary_xi}) yields upper bounds on $\rho(m)$ and $\xi(m)$, that can be used to upper bound the error in (\ref{Eq_corollary}). The following lemma may facilitate the evaluation of a bound on the Gaussian mean width in some cases.

\begin{lemma}
\label{lemma_w_combined_gen}
Let the union of subspaces $\mathcal{U}$ be the sum of $N$ unions of subspaces $\{\mathcal{U}_i\}_{i=1}^N$, i.e., $\mathcal{U} = \sum \limits_{i=1}^{N} \mathcal{U}_i$. Then, we have
\begin{align}
\label{Eq_w_combined_gen}
w(\mathcal{U} \cap \Bbb S^{n-1}) \leq \sum \limits_{i=1}^{N} w(\mathcal{U}_i \cap \Bbb S^{n-1}).
\end{align}
\end{lemma}

\begin{proof}
Let $\{\mathcal{K}_i\}_{i=1}^N$ be $N$ arbitrary sets, we have the following fact
\begin{align}
\label{Eq_w_sum_of_sets_gen}
w \left (\sum \limits_{i=1}^{N} \mathcal{K}_i \right ) = \Ex_{\g} \left \{ \underset{\substack{\u_i \in \mathcal{K}_i \\ i=1 \ldots N}}{\operatorname{sup}} \langle \g, \sum \limits_{i=1}^{N} \u_i \rangle \right \} = \Ex_{\g} \left \{ \sum \limits_{i=1}^{N} \underset{\u_i \in \mathcal{K}_i}{\operatorname{sup}} \langle \g,\u_i \rangle\right \} = \sum \limits_{i=1}^{N} w(\mathcal{K}_i).
\end{align}
Next, note that $\mathcal{U} \cap \Bbb B^n \subset \sum \limits_{i=1}^{N} (\mathcal{U}_i \cap \Bbb B^n)$, and therefore we have
\begin{align}
\label{Eq_w_combined_gen2}
w(\mathcal{U} \cap \Bbb B^n) \leq w \left (\sum \limits_{i=1}^{N} (\mathcal{U}_i \cap \Bbb B^n) \right ).
\end{align}
Finally, using (\ref{Eq_w_sum_of_sets_gen}) together with (\ref{Eq_w_combined_gen2}), and exploiting the property $w(\mathcal{U}_i \cap \Bbb B^n) = w(\mathcal{U}_i \cap \Bbb S^{n-1})$ (see the explanation below (\ref{Eq_lemma3_main2_2}) in Appendix \ref{AppendixC}), we get (\ref{Eq_w_combined_gen}). 
\end{proof}

A direct result of (\ref{Eq_U_B_def2}) and Lemma \ref{lemma_w_combined_gen} is 
\begin{align}
\label{Eq_B_w_U1}
w(\mathcal{U}^B \cap \Bbb S^{n-1}) \leq B \cdot w(\mathcal{U}^1 \cap \Bbb S^{n-1}).
\end{align}
However, in many specific models, $\mathcal{U}^B$ is very similar to $\mathcal{U}^1$, and a bound on $w(\mathcal{U}^B \cap \Bbb S^{n-1})$ which is evaluated directly, can be tighter than (\ref{Eq_B_w_U1}). Nevertheless, Lemma \ref{lemma_w_combined_gen} is beneficial when a combined model is being considered, as discussed below.

The computational complexity of GCoSaMP is dominated by the complexity of the subspace selections problems.
For certain choices of $\mathcal S$, optimal subspace selection methods exist and can be implemented efficiently, while for other choices of $\mathcal S$, relaxations of GCoSaMP are required in order to obtain practical algorithms. These relaxed algorithms may not possess the recovery guarantees from Section \ref{Sec3}. However, they are still closely related to GCoSaMP. 

In the sequel, we discuss various specific models that fall into the general union of subspaces framework, and mainly focus on the following two subjects: Evaluation of the Gaussian mean width and relaxations of GCoSaMP.

\subsection{Sparsity and Structured Sparsity Models}

As discussed in Section \ref{Sec3}, choosing $\mathcal S$ as the set of $\binom{n}{k}$ subspaces, where each subspace is spanned by a different selection of $k$ columns of $\I_n$, we have that $\mathcal{U}^B \equiv  \mathcal{U}^B_{k\text{-sparse}}$ is the set of $Bk$-sparse vectors in $\Bbb R^n$. The quantity $w(\mathcal U^{B}_{k\text{-sparse}} \cap \Bbb S^{n-1})$ is bounded in (\ref{Eq_w_sparse_case}). Note that the upper bound in (\ref{Eq_w_sparse_case}) is tighter than the one obtained by (\ref{Eq_B_w_U1}).
For this choice of $\mathcal S$, GCoSaMP reduces to CoSaMP \cite{needell2009cosamp}: The subspace selection problems, i.e., determining $\tilde{\mathcal{V}}^t$ and $\mathcal{V}^t$, involve finding the best $2k$-sparse or $k$-sparse approximations to a given vector, and thus reduce to support selections that can be calculated easily using thresholding.

Another instance of union of subspaces is structured sparsity, in which the signal model is restricted to predefined sparsity patterns. Thus, assuming that the total sparsity is $k$, the set $\mathcal S$ includes only part of the $\binom{n}{k}$ subspaces associated with the unconstrained $k$-sparse vectors. For certain sparsity patterns, such as block-sparsity or tree-sparsity, there exist efficient methods to obtain a solution to the subspace selection problem \cite{baraniuk2010model}. 
In order to evaluate the Gaussian mean width for these patterns, we make use of the following lemma, which is an extension of Lemma 2.3 in \cite{plan2013robust}.

\begin{lemma}
\label{lemma_w_sparse_struct}

Let the set $\mathcal{S}$ be a subset of the $\binom{n}{k}$ subspaces associated with the unconstrained $k$-sparse vectors, such that $| \mathcal{S} | \leq \operatorname{exp}(\gamma(k,n))$, where $\gamma(k,n)$ is some function of $k$ and $n$. Then, we have
\begin{align}
\label{Eq_w_sparse_struct_case}
w^2(\mathcal{U}^1 \cap \Bbb S^{n-1}) \leq C \cdot \mathcal{O}(k+2\gamma(k,n)).
\end{align}
\end{lemma}

\begin{proof}
Following the proof in \cite{plan2013robust} (which is given in more details in \cite{website2013robusttt}), but replacing $|\mathcal{S}| = \binom{n}{k} \leq \left ( \frac{\mathrm{e}n}{k} \right )^k = \operatorname{exp}(k+k\operatorname{log}\frac{n}{k})$ with $| \mathcal{S} | \leq \operatorname{exp}(\gamma(k,n))$, it is straightforward to obtain (\ref{Eq_w_sparse_struct_case}).
\end{proof}

For tree-sparsity we have $|\mathcal{S}| \leq \frac{(2\mathrm{e})^k}{k+1} = \operatorname{exp}(k+k\operatorname{log}\frac{2}{(k+1)^{1/k}})$ \cite{baraniuk2010model}, i.e., $\gamma(k,n) = k+k\operatorname{log}\frac{2}{(k+1)^{1/k}}$, which yields
\begin{align}
\label{Eq_w_sparse_struct_tree}
w^2(\mathcal{U}^1_{k\text{-tree-sparse}} \cap \Bbb S^{n-1}) \leq C k.
\end{align}

For block-sparsity, where the blocks are disjoint and of size $J$, we have $|\mathcal{S}| = \binom{n/J}{k/J} \leq \left ( \frac{\mathrm{e}n}{k} \right )^{k/J} = \operatorname{exp}(k/J+k/J \operatorname{log}(n/k))$, i.e., $\gamma(k,n) = k/J+k/J \operatorname{log}(n/k)$, which yields
\begin{align}
\label{Eq_w_sparse_struct_tree}
w^2(\mathcal{U}^1_{(k,J)\text{-block-sparse}} \cap \Bbb S^{n-1}) \leq C \cdot \mathcal{O}(k+\frac{k}{J} \operatorname{log}\frac{n}{k}).
\end{align}

We remind the reader that we may calculate $w(\mathcal{U}^B \cap \Bbb S^{n-1})$ from $w(\mathcal{U}^1 \cap \Bbb S^{n-1})$ using (\ref{Eq_B_w_U1}). 

\subsection{Low-Rank Model}
Note that using (column or row stack) vectorization, our theory can be applied immediately to matrices. Assuming that the signal is a rank-$r$ matrix $\X \in \Bbb{R}^{n_1 \times n_2}$ (now the signal dimension is $n=n_1 \cdot n_2$), using singular value decomposition (SVD), $\X$ can be written as
\begin{align}
\label{Eq_X_svd_norm}
\X = \U\bSigma\V^* = \sum \limits_{i=1}^{r} \sigma_i \u_i \v_i^*,
\end{align}
where $\u_i, \v_i$ are the columns of the orthonormal matrices $\U$ and $\V$, respectively, associated with nonzero singular value $\sigma_i$. It is easy to see that all the rank-$r$ matrices that can be written as a linear combination of the same $r$ rank-1 matrices $\{\u_i \v_i^*\}_{i=1}^{r}$ lie on the subspace $\left \{ \sum \limits_{i=1}^r a_i \u_i \v_i^* : a_i \in \Bbb R \right \}$. Therefore, we can choose $\mathcal S$ as the infinite set of such subspaces, each spanned by $r$ rank-1 matrices, and the set of all rank-$r$ matrices can be seen as an infinite union of these subspaces. 
As a result from this choice of $\mathcal S$, the set $\mathcal S^B$ is also an infinite set of subspaces, where each subspace is spanned by no more than $Br$ rank-1 matrices, and GCoSaMP reduces to ADMiRA \cite{lee2010admira}.
In this case, the subspace selection problems involve finding the best rank-$2r$ or rank-$r$ approximations to a given matrix, which can be done efficiently by truncating the SVD of the matrix.
We turn to evaluate the Gaussian mean width for $\mathcal{U}^B \equiv  \mathcal{U}^B_{\text{rank-}r}$.
Note that we have the following identity for the Frobenius norm $\| \X \|_F$ 
\begin{align}
\label{Eq_Frob_norm}
\| \X \|_F = \sqrt{\mathrm{tr}(\X^*\X)} = \sqrt{\mathrm{tr}(\V\bSigma^2\V^*)} =  \sqrt{\mathrm{tr}(\bSigma^2)} = \| \mathrm{diag}(\bSigma)\|_2,
\end{align}
where $\mathrm{diag}$ denotes the main diagonal of a matrix. Therefore, applying Cauchy-Schwarz inequality to the singular values of $\X$, we have that the set $\mathcal{U}^B_{\text{rank-}r} \cap \Bbb S^{n-1}$ is contained in the convex set
\begin{align}
\label{Eq_K_n1n2k_def}
\mathcal K_{n_1,n_2,Br} \triangleq \{ \X \in \Bbb R^{n_1 \times n_2} \, : \, \|\X\|_* \leq \sqrt{Br}, \|\X\|_F \leq 1 \},
\end{align}
where $\| \X \|_*$ denotes the nuclear norm, i.e., the sum of the singular values of $\X$. Therefore, the upper bound on $w(\mathcal K_{n_1,n_2,Br})$ in \cite{plan2013robust} also upper bounds the Gaussian mean width of the set we are interested in
\begin{align}
\label{Eq_w_low_rank_case}
w(\mathcal{U}^B_{\text{rank-}r} \cap \Bbb S^{n-1}) \leq (\sqrt{n_1}+\sqrt{n_2})\sqrt{Br}.
\end{align}

\subsection{Sparse Synthesis Model}

Other choices of $\mathcal S$ may require relaxation of GCoSaMP in order to reduce its computational complexity. For example, \cite{davenport2013signal, giryes2015greedy} considered the synthesis model, which assumes that the signal $\x$ admits a $k$-sparse representation $\balpha$ in a given dictionary $\D \in \Bbb R^{n \times d}$, i.e., $\x=\D\balpha$, where $\balpha$ is a $k$-sparse vector in $\Bbb R^d$ (typically $d>n$). Denote the columns of $\D$ by $\{\d_i\}_{i=1}^d$. For this model, $\mathcal S$ consists of $\binom{d}{k}$ subspaces, where each subspace is spanned by a different selection of $k$ vectors from $\{\d_i\}_{i=1}^d$. The set $\mathcal S^B$ consists of subspaces spanned by no more than $Bk$ vectors from $\{\d_i\}_{i=1}^d$. In this case, the subspace selection problem is NP-hard in general \cite{davis1997adaptive}.
For this reason, both works suggested a signal space CoSaMP (SSCoSaMP) algorithm, in which the optimal subspace selections are replaced with approximated methods. While the theoretical results in \cite{davenport2013signal, giryes2015greedy} require that these approximations satisfy strict near-optimality constraints, both works also empirically examine the performance of SSCoSaMP using practical (but not theoretically backed) selections methods such as hard thresholding or OMP.
Let us upper bound $w(\mathcal U^B_{k\text{-synthesis}} \cap \Bbb S^{n-1})$ for a dictionary $\D$ that has the RIP of order $Bk$ with an RIP-constant $\delta_{Bk}$. It is easy to show that if $\balpha$ satisfies $\|\balpha\|_0=Bk$ and $\|\D\balpha\|_2=1$, it also satisfies $\|\balpha\|_1 \leq \sqrt{\frac{Bk}{1-\delta_{Bk}}}$. Now, define the following convex set (for fixed $\balpha$ and $\D$)
\begin{align}
\label{Eq_K_alpha}
\mathcal K_{\balpha,\D} \triangleq \|\balpha\|_1 \cdot \mathrm{conv}\{\pm\d_i\}_{i=1}^d,
\end{align}
where $\mathrm{conv}$ denotes the convex hull of a set. It was shown in \cite{vershynin2015estimation} that
\begin{align}
\label{Eq_w_K_alpha}
w(\mathcal K_{\balpha,\D}) \leq C \|\balpha\|_1 \sqrt{\mathrm{log}d},
\end{align}
where $C$ is a positive constant. Therefore, since $\x = \sum \limits_{i=1}^{d} \alpha_i \d_i$ leads to $\x \in \mathcal K_{\balpha,\D}$, from the arguments above, we have
\begin{align}
\label{Eq_w_dictionary_case}
w(\mathcal U^B_{k\text{-synthesis}} \cap \Bbb S^{n-1}) \leq C \sqrt{\frac{Bk}{1-\delta_{Bk}}\mathrm{log}d}.
\end{align}

\subsection{Cosparse Analysis Model}

Another case in which the subspace selection problems are NP-hard in general is the cosparse analysis model \cite{nam2013cosparse, tillmann2014projection}.
The cosparse analysis model assumes that the analyzed vector $\bOmega\x$ is sparse, where $\bOmega \in \Bbb R^{p \times n}$ is a possibly redundant analysis operator ($p \geq n$). However, it focusses on the cosupport of $\bOmega\x$, i.e., the zeros of $\bOmega\x$, instead of on its support. Denote the rows of $\bOmega$ by $\{\bomega_i^T\}_{i=1}^p$, the cosupport of $\bOmega\x$ by $\Lambda$, and the size of the cosupport by $\ell$. We have that $\x \in \mathcal{V}_{\Lambda}$, where 
\begin{align}
\mathcal{V}_{\Lambda} \triangleq \mathrm{span}(\bomega_i, i \in \Lambda)^{\perp} = \{\x \, : \, \langle \x,\bomega_i \rangle =0, \forall i \in \Lambda \}.
\end{align}
However, since it is assumed that only the size of the cosupport is known, and not the cosupport itself, the set $\mathcal S$ consists of all the $\binom{p}{\ell}$ possible subspaces $\mathcal{S} = \{\mathcal{V}_{\Lambda} \,:\, |\Lambda|=\ell  \}$, and $\x$ resides in a union of these subspaces.
An analysis version of CoSaMP (ACoSaMP) was proposed in \cite{giryes2014greedy}.
Even though for some analysis operators, such as the 1D finite difference operator or a concatenation of the identity operator and the 1D finite difference operator, the optimal cosupport selection (or equivalently, subspace selection) can be accomplished in polynomial complexity, in general, no efficient method for optimal cosupport selection is known.
Therefore, ACoSaMP replaces optimal cosupport selections with approximated methods. However, except this relaxation, ACoSaMP can be seen as an instance of GCoSaMP. In each iteration, it extracts $\Lambda_{\Delta}^t$, a near-optimal cosupport of size $2\ell-p$ of the proxy $\A^*\r$. This is equivalent to finding a near-optimal subspace $\mathcal{V}_{\Lambda_{\Delta}^t} \in \mathcal{S}^2$, since the minimal size of the cosupport of a sum of two $\ell$-cosparse vectors is $2\ell-p$ (see \cite{giryes2014greedy} for more details). Next, the cosupport for the least squares (LS) estimation step is obtained using intersection of the newly extracted cosupport and the cosupport of the previous iteration, i.e., $\tilde{\Lambda}^t = \Lambda^{t-1} \cap \Lambda_{\Delta}^t$. This is equivalent to constraining the LS estimation to $\mathcal{V}_{\Lambda^t} + \mathcal{V}_{\Lambda_{\Delta}^t}$. In practice, efficient methods to obtain near-optimal cosupports are also unknown (for most analysis operators). Therefore, \cite{giryes2014greedy} empirically examined the performance of ACoSaMP using simple thresholding. Nevertheless, the algorithm exhibited good performance.
Assuming that $\bOmega$ has linearly independent columns and denoting its pseudoinverse by $\bOmega^{\dagger}$, the quantity $w(\mathcal U^B_{\ell\text{-analysis}} \cap \Bbb S^{n-1})$ can be evaluated by exploiting results for the sparse synthesis model with an $n \times p$ dictionary $\D = \bOmega^{\dagger}$, since $\x = \bOmega^{\dagger}\bOmega\x$ and $\bOmega\x$ is a $(p-\ell)$-sparse vector. Therefore, if $\bOmega^{\dagger}$ has the RIP of order $B(p-\ell)$ with an RIP-constant $\delta_{B(p-\ell)}$, we have
\begin{align}
\label{Eq_w_dictionary_case_cosparse}
w(\mathcal U^B_{\ell\text{-analysis}} \cap \Bbb S^{n-1}) \leq C \sqrt{\frac{B(p-\ell)}{1-\delta_{B(p-\ell)}}\mathrm{log}p}
\end{align}
for the same reasons that yield (\ref{Eq_w_dictionary_case}).

\subsection{Sparsity and Low-Rank Combined Model}

In the following, we discuss other choices of $\mathcal S$ for which the optimal subspace selections are unknown in general. Specifically, we consider the case of a combined model. 
One such example is recovering a matrix $\X = \L+\S$, that is the sum of a rank-$r$ matrix, $\L$, and a $k$-sparse matrix, $\S$. In this case, each subspace in the set $\mathcal S$ is the sum of two ”componential subspaces”, one associated with $k$-sparse matrices and the other associated with rank-$r$ matrices (as discussed above). As a result, each subspace in the set $\mathcal{S}^B$ is a sum of two componential subspaces as well, one associated with matrices whose sparsity does not exceed $Bk$, and the other associated with matrices whose rank does not exceed $Br$.
In \cite{waters2011sparcs}, an algorithm named SpaRCS was proposed for this problem. This algorithm resembles the general algorithm that we present here, except that in SpaRCS the subspace selection steps and the LS estimation step are performed separately for each componential subspace. Although lacking theoretical performance guarantees, SpaRCS has shown good performance in experiments. Splitting the subspace selection allows performing hard thresholding for the sparse part and truncated SVD for the low-rank part, that can be seen as an approximation to the unknown optimal subspace selection.
Denoting the union of subspaces for this model by $\mathcal{U}^B_{k\text{-sparse, rank-}r}$, we have $\mathcal{U}^B_{k\text{-sparse, rank-}r} = \mathcal{U}^B_{k\text{-sparse}} + \mathcal{U}^B_{\text{rank-}r}$. Therefore, using Lemma \ref{lemma_w_combined_gen} we get
\begin{align}
\label{Eq_w_sparse_low_rank_case}
w(\mathcal{U}^B_{k\text{-sparse, rank-}r} \cap \Bbb S^{n-1}) \leq w(\mathcal{U}^B_{k\text{-sparse}} \cap \Bbb S^{n-1}) + w(\mathcal{U}^B_{\text{rank-}r} \cap \Bbb S^{n-1}),
\end{align}
where the terms in the right-hand side of the last inequality are bounded in (\ref{Eq_w_sparse_case}) and (\ref{Eq_w_low_rank_case}). However, note that in order to determine that an algorithm has the guarantees from Section \ref{Sec3}, it must be an exact instance GCoSaMP. Such an algorithm is still unknown in this case, as optimal subspace selection methods are unknown.

\subsection{Sparse Synthesis and Cosparse Analysis Combined Model}

Another interesting combined model is the one that uses both the sparse synthesis model and the cosparse analysis model,
and
considers a signal $\x = \x_1 + \x_2$, where $\x_1$ admits a $k$-sparse representation $\balpha$ in a given dictionary $\D \in \Bbb R^{n \times d}$, and $\x_2$ is $\ell$-cosparse under a given analysis operator $\bOmega \in \Bbb R^{p \times n}$, i.e., $\bOmega\x_2$ has at least $\ell$ zeros. This combined model may be advantageous for images that contain both texture and piecewise constant (cartoon) parts. The texture part may admit a sparse representation in a suitable dictionary, and the cartoon part is expected to have a cosparse analysis representation under the finite difference operator, even if the original image does not comply with each of these models separately. Clearly, this is the case when the texture and cartoon parts overlap.

Denoting the union of subspaces for this model by $\mathcal{U}^B_{k\text{-synthesis, }\ell\text{-analysis}}$, we have $\mathcal{U}^B_{k\text{-synthesis, }\ell\text{-analysis}} = \mathcal{U}^B_{k\text{-synthesis}} + \mathcal{U}^B_{\ell\text{-analysis}}$. Therefore, using Lemma \ref{lemma_w_combined_gen} we get
\begin{align}
\label{Eq_w_snythesis_anlysis_case}
w(\mathcal{U}^B_{k\text{-synthesis, }\ell\text{-analysis}} \cap \Bbb S^{n-1}) \leq w(\mathcal{U}^B_{k\text{-synthesis}} \cap \Bbb S^{n-1}) + w(\mathcal{U}^B_{\ell \text{-analysis}} \cap \Bbb S^{n-1}),
\end{align}
where the terms in the right-hand side of the last inequality are bounded in (\ref{Eq_w_dictionary_case}) and (\ref{Eq_w_dictionary_case_cosparse}) for $\D$ and $\bOmega^{\dagger}$ that have the RIP.
However, as follows from our discussion on each sub-model, the subspace selection problem for the combined model is NP-hard in general. Therefore, we propose a practical reconstruction algorithm for this model, inspired by GCoSaMP.

We propose synthesis-analysis CoSaMP (SACoSaMP) for reconstruction of $\x$ from $m$ ($m < n$) compressive linear measurements given by $\y = \A \x + \e$. The algorithm is presented in Algorithm 2. It can also be used to directly recover $\x_1$ and $\x_2$ separately. It combines SSCoSaMP \cite{davenport2013signal} and ACoSaMP \cite{giryes2014greedy} in a way that resembles the way SpaRCS combines CoSaMP \cite{needell2009cosamp} and ADMiRA \cite{lee2010admira}, except that the LS estimation step in SACoSaMP is unified, a fact that can be considered as being more loyal to GCoSaMP, and is justified by our experiments in Section \ref{Sec5b}. Let us clarify the notations being used in Algorithm 2. We use $[1..p]$ to denote the integers between $1$ and $p$, $\operatorname{supp}(\u,k)$ for the indices of the $k$ largest elements (in magnitude) of the vector $\u$, and $\operatorname{cosupp}(\u,\ell)$ for the indices of the $\ell$ smallest elements (in magnitude) of the vector $\u$. The notation $\balpha|_T$ stands for a sub-vector of $\balpha$ with elements that belong to the set $T$, and $T^C$ denotes the complementary set of $T$. $\D_T$ is a sub-matrix of $\D$ with columns corresponding to the set $T$, and by abuse of notation, $\bOmega_{\Lambda}$ is a sub-matrix of $\bOmega$ with rows that belong to the set $\Lambda$. Finally, $\Q_{\Lambda}=\I_n-\bOmega_{\Lambda}^{\dagger}\bOmega_{\Lambda}$ is the orthogonal projection onto the orthogonal complement of the row space of $\bOmega_{\Lambda}$.
We note that the extracted cosupport $\Lambda_{\Delta}^t$ is of size $\ell$ rather than $2\ell-p$, since it was empirically shown in \cite{giryes2014greedy} that this modification improves the performance of ACoSaMP.

SACoSaMP is a relaxed version of GCoSaMP. The support and cosupport selections are separated, and use simple thresholding, which is not optimal in general, and thus the resulted subspace may not be optimal.
Therefore, even if one can evaluate the Gaussian mean width for the combined model described above, SACoSaMP does not possess the recovery guarantees from Section \ref{Sec3}. However, it is still inspired by the general scheme of GCoSaMP and by the idea of not being restricted to a single traditional model. The usefulness of the algorithm is demonstrated in Section \ref{Sec5b}.

\begin{algorithm}
\caption{Synthesis-Analysis CoSaMP (SACoSaMP)}
\vspace{2mm}
\kwInput{$\A, \y, k, \ell, \D, \bOmega,$ stopping criterion, where $\y = \A(\x_1+\x_2)+\e$, such that $\e$ is an additive noise, $k$ is the sparsity of $\x_1$ in $\D$ and $\ell$ is the cosparsity of $\x_2$ in $\bOmega$.}
\kwOutput{$\hat{\x}_1, \hat{\x}_2, \hat{\x}$ estimates for $\x_1, \x_2, \x$, respectively.}
\kwInitialize{$\r=\y, \x_1^0=0, \x_2^0=0, t=0, T^0=\varnothing, \Lambda^0=[1..p]$}
\While{stopping criterion not met}{
    $t = t+1$\;
    $\tilde{\v} = \A^*\r$\;
    $T_{\Delta}^t = \operatorname{supp}(\D^*\tilde{\v},2k)$; \,\, $\Lambda_{\Delta}^t = \operatorname{cosupp}(\bOmega\tilde{\v},\ell)$\;
    $\tilde{T}^t = T^{t-1} \cup T_{\Delta}^t$; \,\, $\tilde{\Lambda}^t = \Lambda^{t-1} \cap \Lambda_{\Delta}^t$\;
    $(\tilde{\balpha}, \tilde{\x}_2) = \underset{\overline{\balpha},\overline{\x}_2}{\operatorname{argmin}} \| \y - \A(\D\overline{\balpha}+\overline{\x}_2) \|_2 \,\,\, \operatorname{s.t.} \,\,\, \overline{\balpha}|_{(\tilde{T}^t)^C}=0, \,\, \bOmega_{\tilde{\Lambda}^t}\overline{\x}_2=0$\;
    $T^t = \operatorname{supp}(\tilde{\balpha},k)$; \,\, $\Lambda^t = \operatorname{cosupp}(\bOmega\tilde{\x}_2,\ell)$\;
    $\x_1^t = \D_{T^t}\tilde{\balpha}_{T^t}$; \,\, $\x_2^t = \Q_{\Lambda^t}\tilde{\x}_2$\;
    $\r = \y - \A (\x_1^t+\x_2^t)$\;
}
$\hat{\x}_1=\x_1^t; \,\, \hat{\x}_2=\x_2^t; \,\, \hat{\x} = \x_1^t + \x_2^t$\;
\end{algorithm}


\section{Experiments}
\label{Sec5}

The experiments section is divided into two parts. In the first part, we conduct an experiment that supports the theoretical result on the denoising capabilities of GCoSaMP. Specifically, we focus on the original sparsity-based CoSaMP, which is equivalent to GCoSaMP when $\mathcal S$ is the set of subspaces associated with the traditional sparsity model.
In the second part, we consider the sparse synthesis and cosparse analysis combined model. We examine the performance of SACoSaMP, focusing on its application to simultaneous image reconstruction and structured noise removal.

\subsection{The vanishing effect of stationary noise in CoSaMP}
\label{Sec5a}

As was mentioned above, when $\mathcal S$ is the set of subspaces associated with the traditional sparsity model, GCoSaMP reduces to CoSaMP. Using synthetic data, we demonstrate that for fixed $k,n$ and large enough $m$, the reconstruction error (after the convergence of the algorithm) decreases proportionally to $1/\sqrt{m}$. This empirical result supports the theoretical result in Section \ref{Sec3}, that for $\|\e\|_2=\mathcal{O}(\sqrt{m})$ (stationary noise) and $m \gg m_0$, the reconstruction error behaves like $\mathcal{O}(\sqrt{\frac{m_0}{m}})$.
We use a signal of length $n=2\mathrm{e}5$ with $k=5$ nonzero entries independently drawn from the standard normal distribution, and zero-mean noise drawn from Laplace distribution. The measurement matrix $\A$ is an $m \times  n$ matrix with $m$ varying between $100$ and $1\mathrm{e}5$. For each value of $m$ we perform 30 Monte Carlo experiments, in which the entries of $\A$ are independently drawn from the standard normal distribution and the noise is scaled such that $\|\e\|_2/\|\A\x\|_2=0.01$. Fig. \ref{fig:vanishing_err}. shows the Euclidean norm of the reconstruction error. The markers represent empirical results and the solid line is computed according to $cm^{-0.5}$, where the constant $c$ compensates for the shift at the last point ($m=1\mathrm{e}5$). Starting from $m=1\mathrm{e}3$, the results are in good agreement, even though our theoretical convergence guarantees required larger values of $m$, since $14.5^2 m_0$ is about $1.18\mathrm{e}4$ (see (\ref{Eq_w_sparse_case})).
 
\begin{figure}
    \centering
    \includegraphics[width=0.5\textwidth]{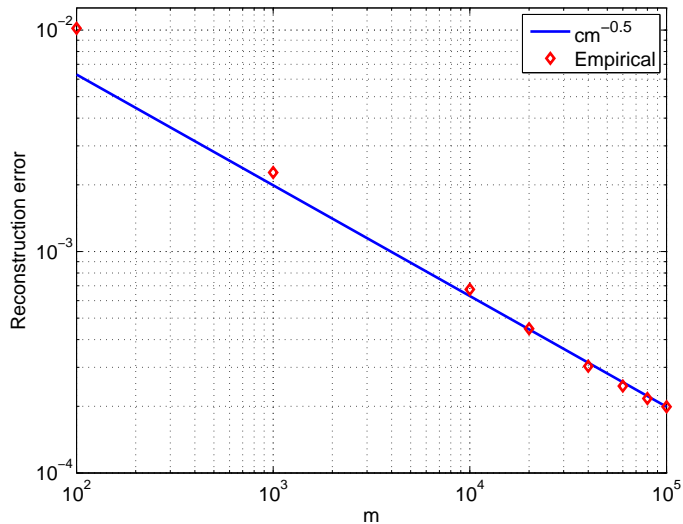}
    \caption{CoSaMP reconstruction error vs. number of measurements} \label{fig:vanishing_err}
\end{figure}

\subsection{Reconstruction and structured noise removal using SACoSaMP}
\label{Sec5b}

We turn now to examine the performance of SACoSaMP. We focus on its application to simultaneous image reconstruction and denoising, when the noise is nonstationary and structured. Such noise may result from intermittent interference between electronic components during the image acquisition process. In our experiments we use the sparse synthesis model for the noise. The dictionary $\D$ is the (inverse) local discrete cosine transform (DCT) with window size of $64 \times 64$ pixels and overlap of 32 pixels. The $16 \times 16$ lowest frequencies in each window are not included in the dictionary. The texture image $\x_1$ that represents the noise is obtained using $k=500$ random dictionary words with Gaussian coefficients. The original clean image we consider as $\x_2$ is the $512 \times 512$ natural image {\em house} given in Fig. \ref{fig:house_clean}. The Frobenius norm of the noise image is scaled to 0.1 of the Frobenius norm of the {\em house} image. The noisy image $\x_1+\x_2$, quantized to 8 bits per pixel (bpp), is given in Fig. \ref{fig:house_mixed}. The sampling operator $\A$ is a two dimensional Fourier transform that measures only $m/n=0.254$ of the points in the Fourier domain according to the binary mask given in Fig. \ref{fig:mask_vardens_512}. The cosparse operator $\bOmega$ is the finite difference analysis operator that computes horizontal and vertical discrete derivatives of an image.  The na\"{\i}ve reconstruction obtained by inverse Fourier transform on the measurements is shown in Fig. \ref{fig:house_trivial}. Its peak signal to noise ratio (PSNR) with respect to the noisy image is 24.11 dB. Clearly, it is also not a satisfying reconstruction of $\x_2$.

Figs. \ref{fig:house_clean_rec} and \ref{fig:house_mixed_rec} show the results of SACoSaMP obtained after 6 iterations, using $k=500$ and $\ell=386464$ ($\ell$ is slightly larger than the number of entries with magnitude lower than 3 in the analysis representation of $\x_2$). Fig. \ref{fig:house_clean_rec} shows the reconstruction of $\x_2$, quantized to 8 bpp. Its PSNR, with respect to the clean {\em house} image, is 35.16 dB. One can see that this cosparse part still includes significant portion of the texture of the original {\em house} image, but none of the texture of the noise. Fig. \ref{fig:house_mixed_rec} shows the (full) reconstruction of $\x_1+\x_2$ quantized to 8 bpp. Its PSNR, with respect to the noisy image, is 35.28 dB. 

Regarding the implementation of SACoSaMP, in each iteration we take only the real part of the results and we use conjugate gradients for the minimization process. We stop the algorithm when there is no significant reduction in the residual error $\|\r\|_2$.

Using the same parameters but modifying SACoSaMP by splitting the LS step into two separated LS estimations, one for each sub-model (as done in SpaRCS \cite{waters2011sparcs}), we get inferior results. The number of iterations until convergence is 14, and the reconstructions have smaller PSNR. Fig. \ref{fig:house_clean_rec_sep} shows the reconstruction of $\x_2$. Its PSNR (with respect to the clean image) is 29.34 dB. The reconstruction of $\x_1+\x_2$ has PSNR of 28.78 dB (with respect to the noisy image). We do not show the latter in the paper due to space limitations.
We try other several values of $\ell$, but no significant improvement is obtained.

We also examine the performance of ACoSaMP (which is based only on the cosparse model). We try to reconstruct the clean {\em house} image using large value of $\ell$, and to reconstruct the noisy image using small value of $\ell$. In both cases the results are not satisfying. For example, the result of ACoSaMP with $\ell=386464$ is shown in Fig. \ref{fig:house_acosamp}.


\begin{figure}
 \centering
  \subcaptionbox{{\em House}  \label{fig:house_clean}}{%
  \includegraphics[width=0.24\columnwidth]{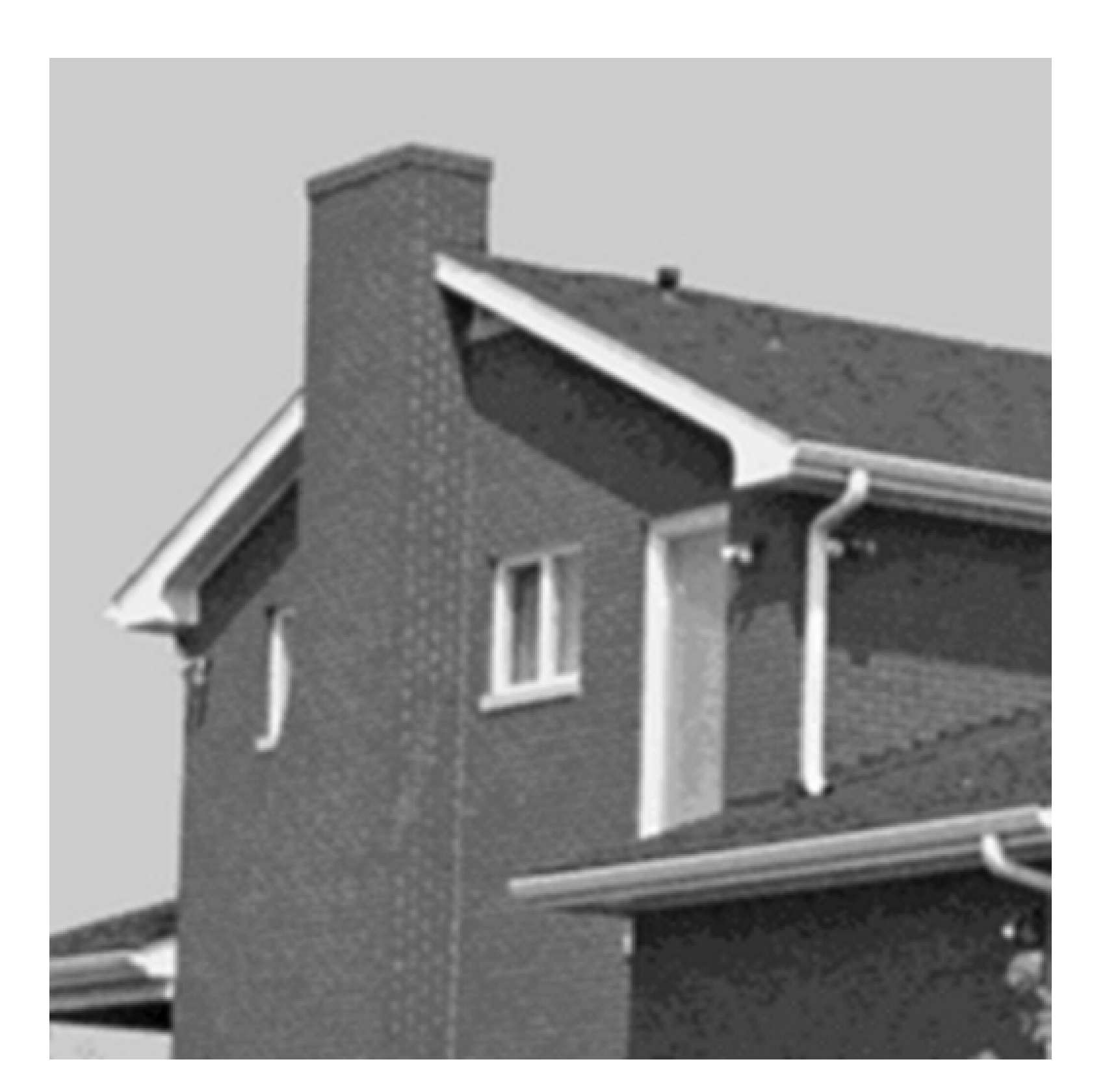}}
  \subcaptionbox{{\em House} + textured noise  \label{fig:house_mixed}}{%
  \includegraphics[width=0.24\columnwidth]{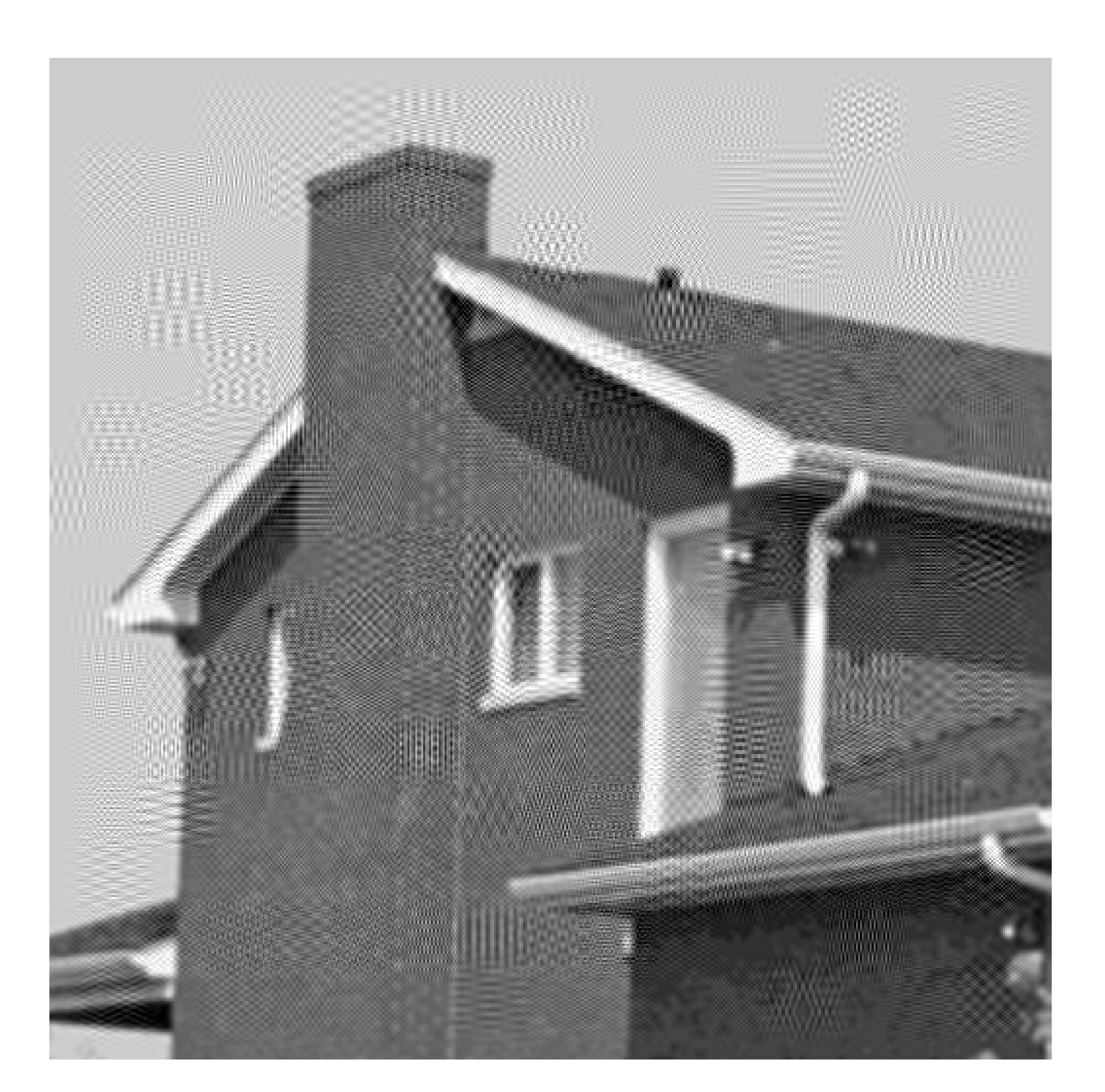}}
  \subcaptionbox{Sampling mask  \label{fig:mask_vardens_512}}{%
  \includegraphics[width=0.24\columnwidth]{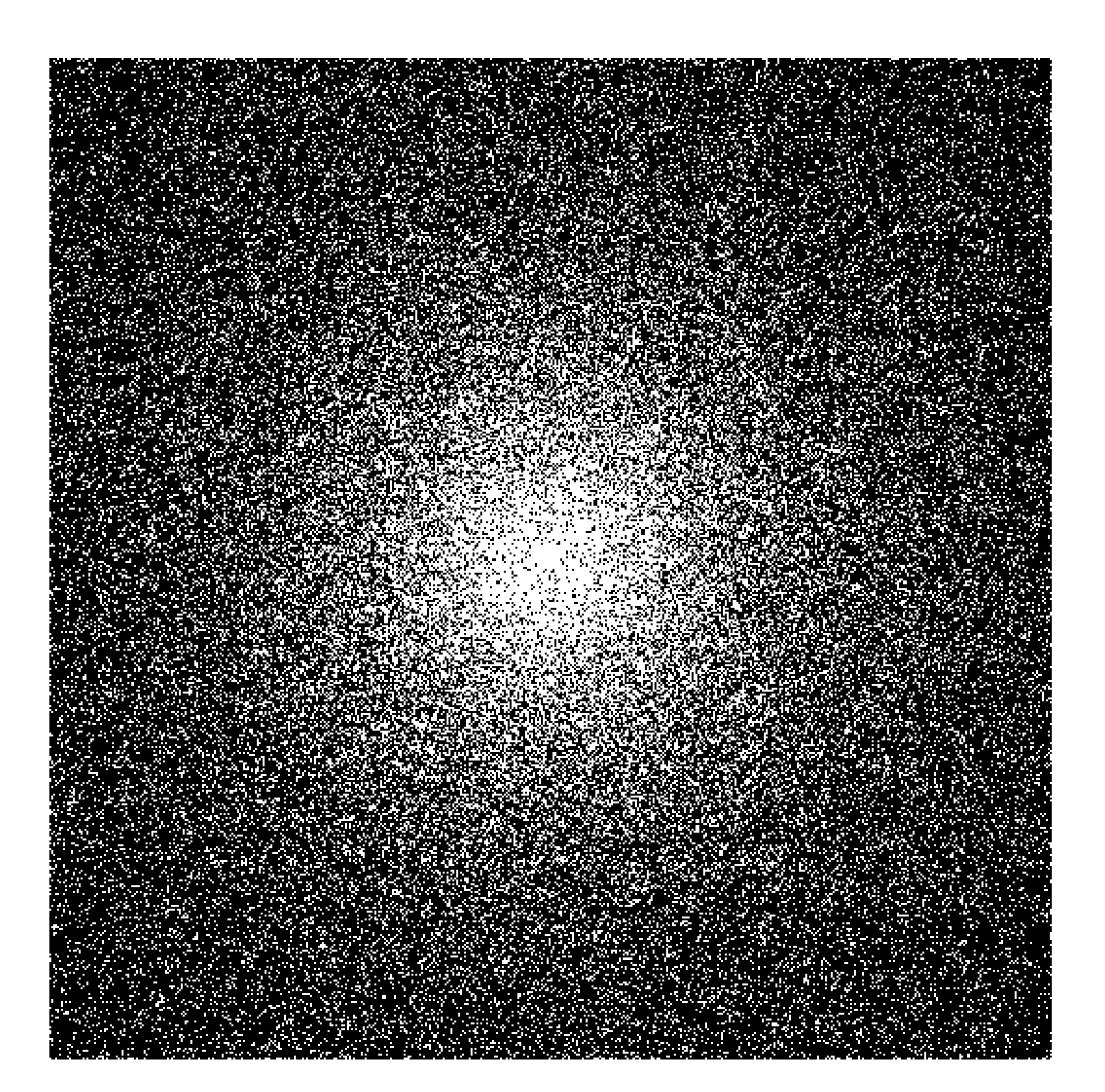}}
  \subcaptionbox{Na\"{\i}ve reconstruction  \label{fig:house_trivial}}{%
  \includegraphics[width=0.24\columnwidth]{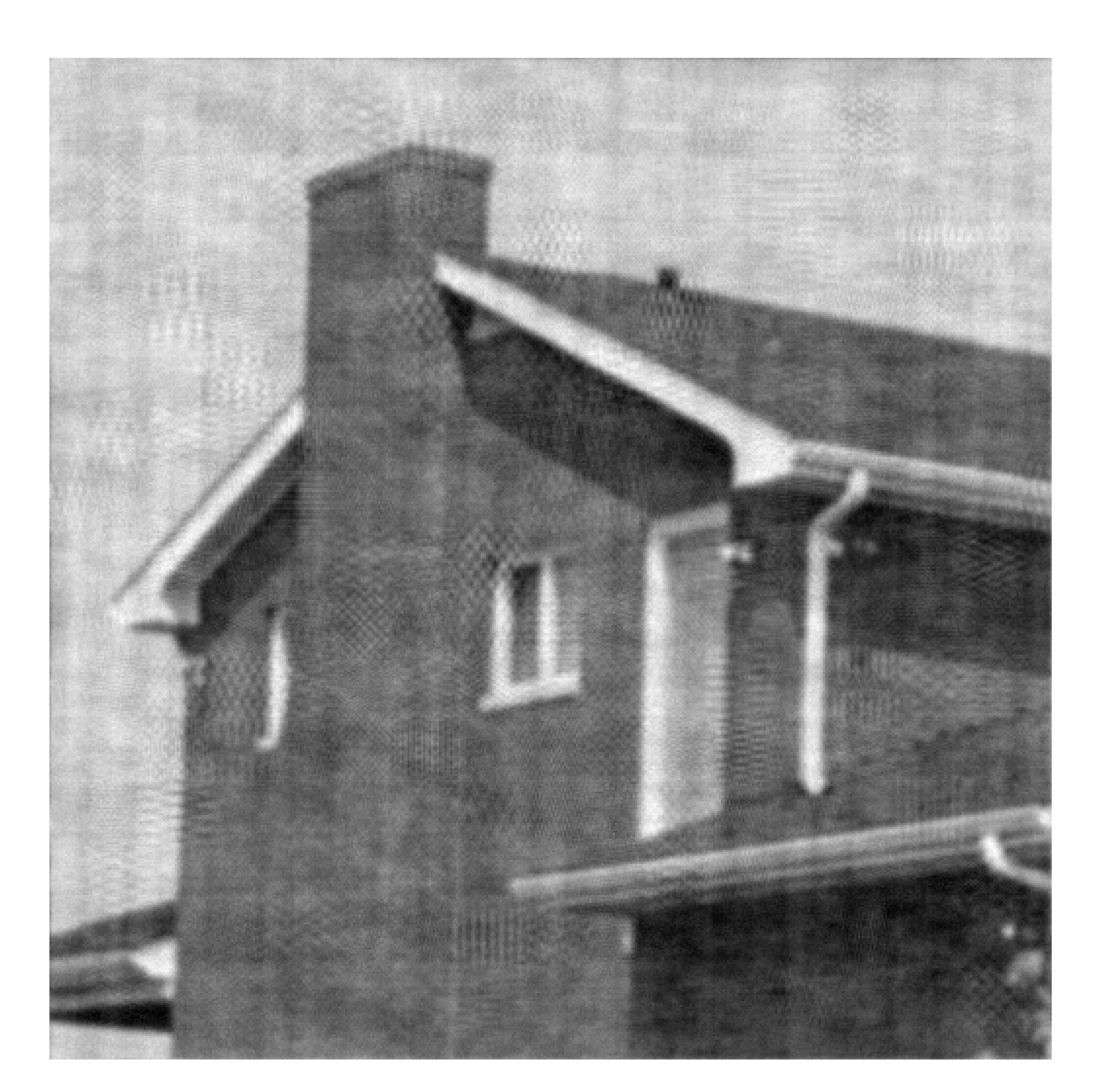}}
  \subcaptionbox{SACoSaMP - {\em house} reconstruction  \label{fig:house_clean_rec}}{%
  \includegraphics[width=0.24\columnwidth]{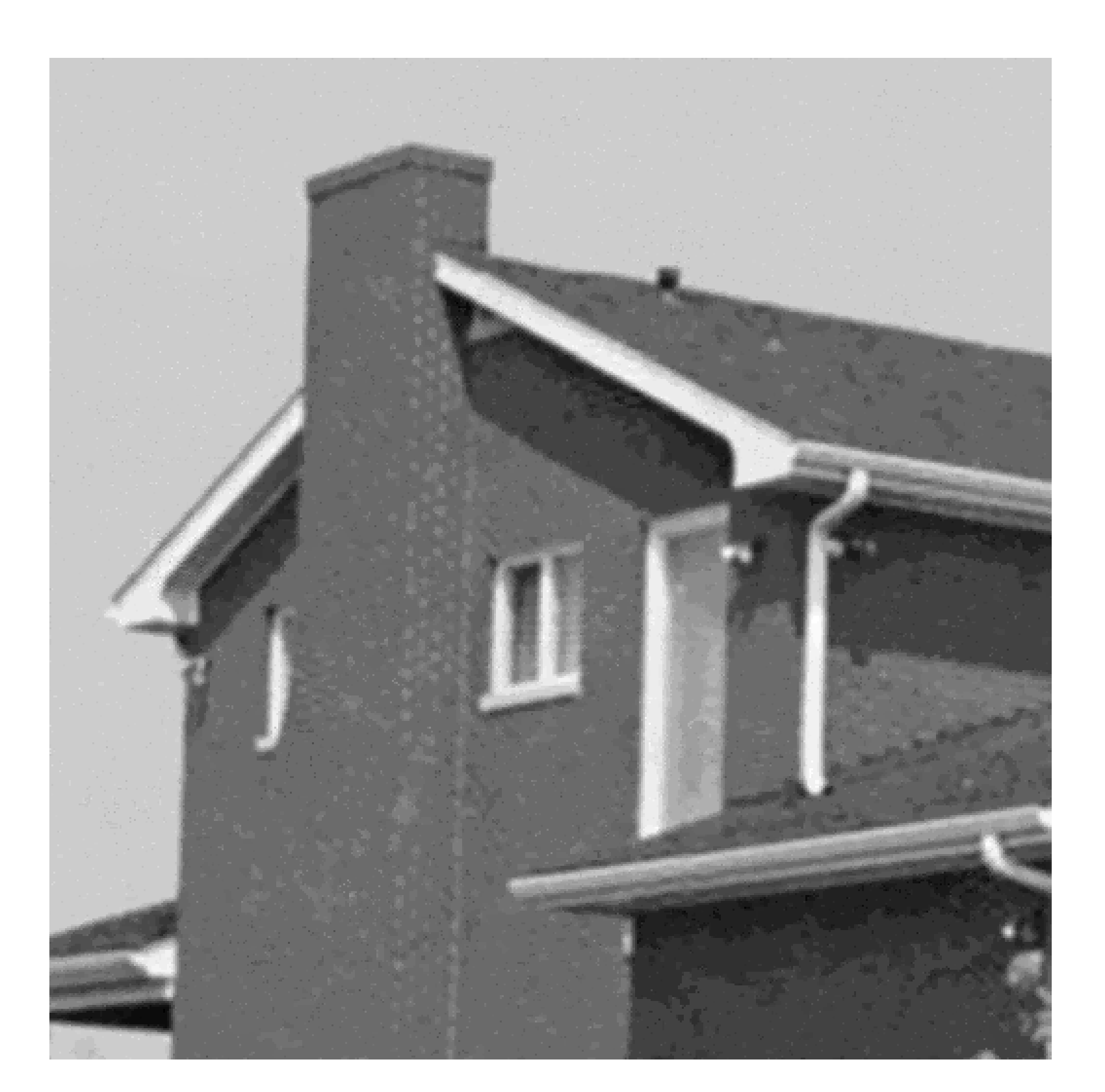}}
  \subcaptionbox{SACoSaMP - {\em house} + textured noise reconstruction \label{fig:house_mixed_rec}}{%
  \includegraphics[width=0.24\columnwidth]{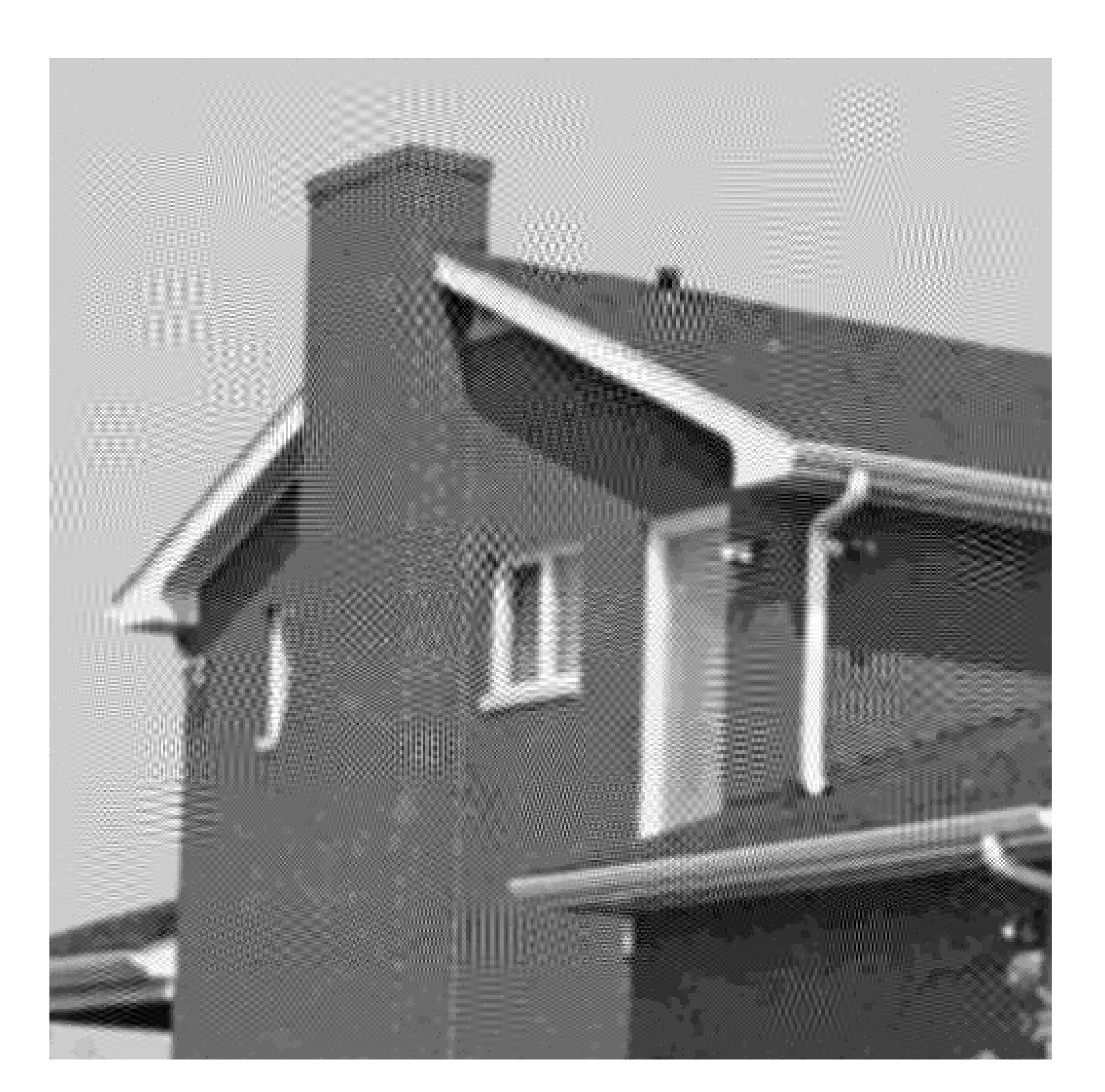}}
  \subcaptionbox{Modified SACoSaMP - {\em house} reconstruction  \label{fig:house_clean_rec_sep}}{%
  \includegraphics[width=0.24\columnwidth]{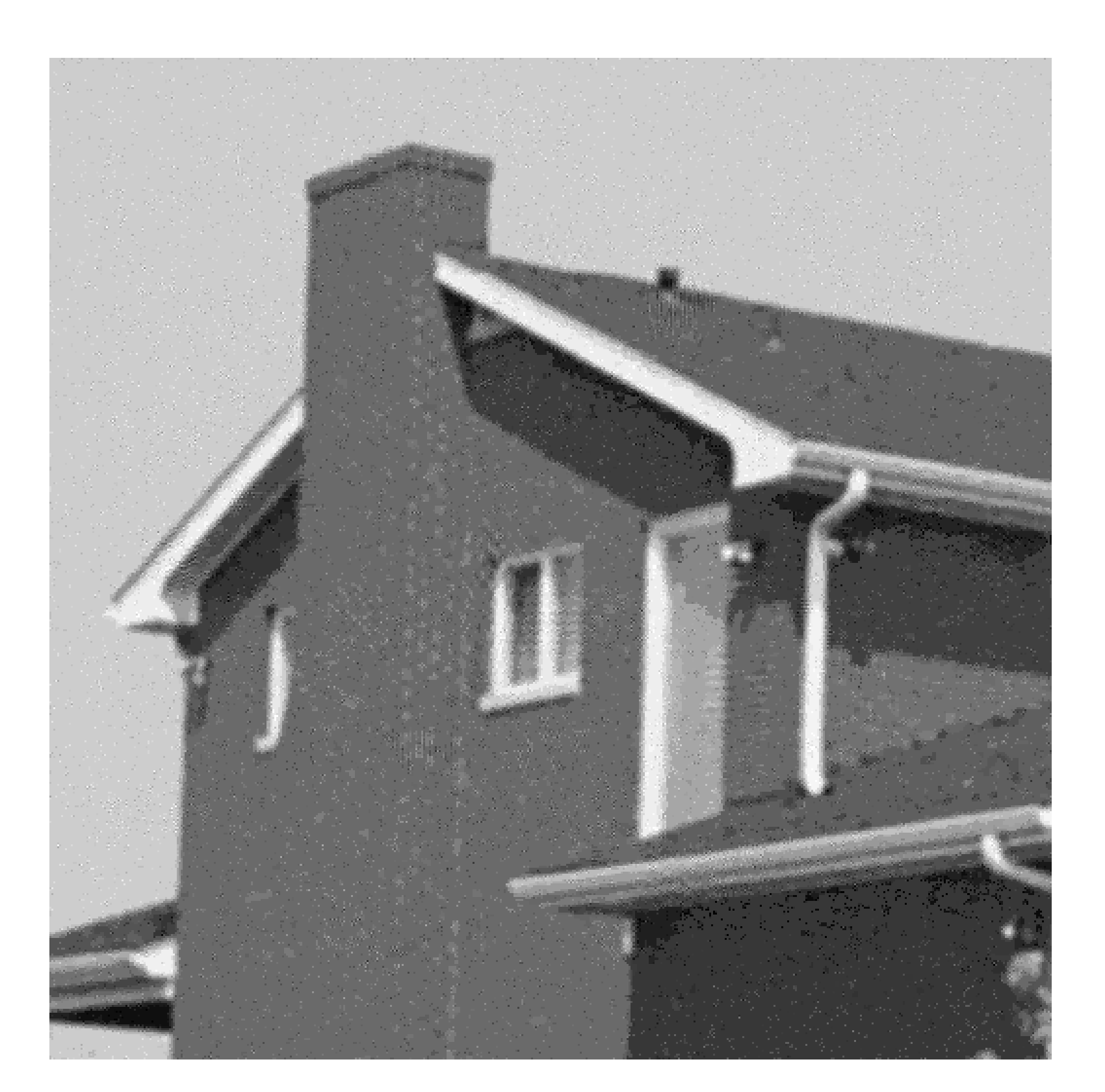}}
  \subcaptionbox{ACoSaMP reconstruction  \label{fig:house_acosamp}}{%
  \includegraphics[width=0.24\columnwidth]{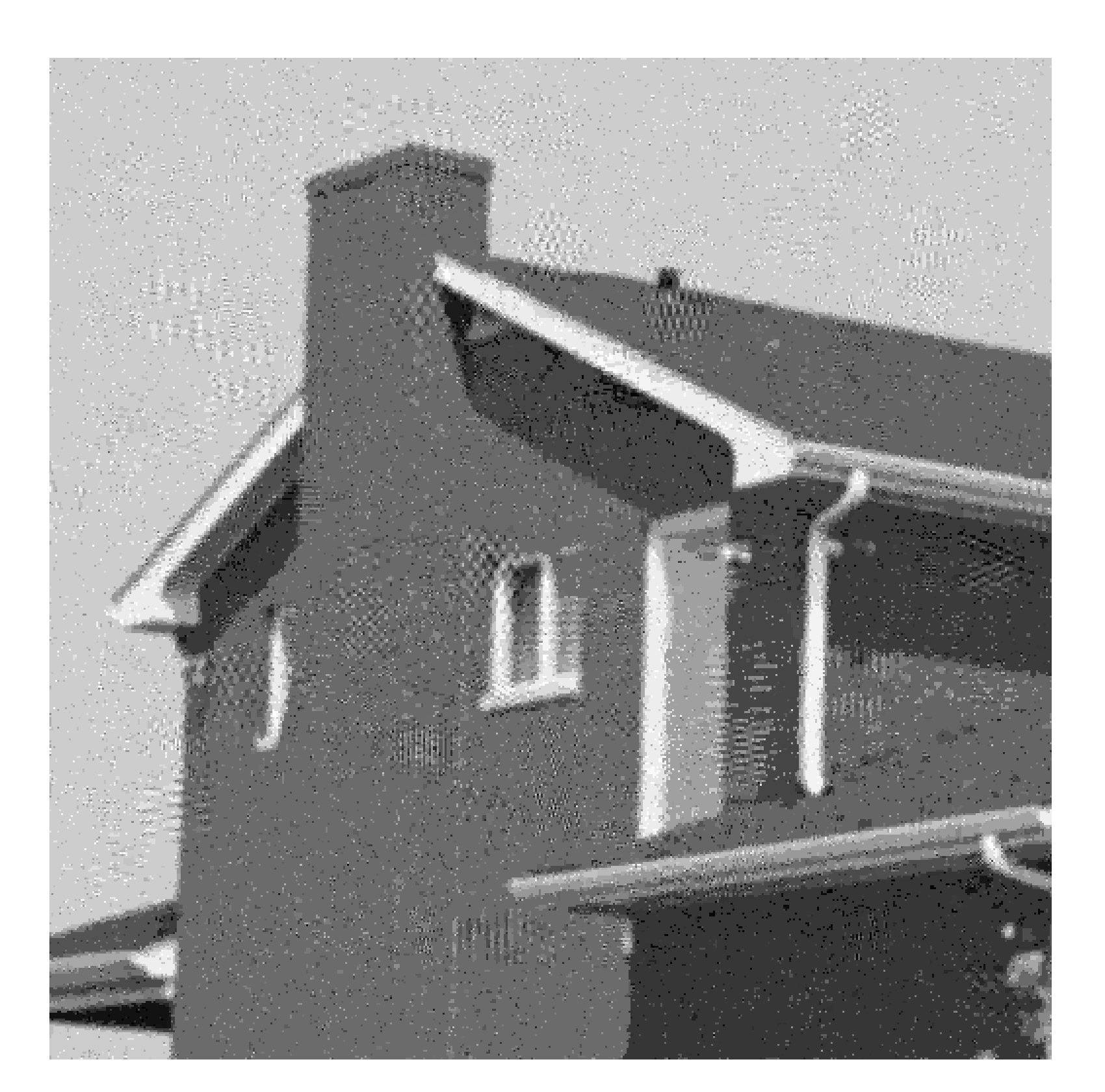}}
  \caption{Experiments inputs and results for {\em house} test image. From left to right and top to bottom: {\em house} image, noisy {\em house} image (textured noise), binary mask for Fourier domain sampling, na\"{\i}ve reconstruction using zero padding and inverse Fourier transform, SACoSaMP reconstruction of {\em house} image (cosparse analysis part), SACoSaMP reconstruction of noisy {\em house} image (cosparse analysis + sparse synthesis parts), modified SACoSaMP (split LS) reconstruction of {\em house} image (cosparse analysis part), and ACoSaMP reconstruction. }\label{fig1}
\end{figure}


We repeat the experiments using a $256 \times 256$ Shepp-Logan phantom image with modified intensities instead of the {\em house} image. The binary mask consists of 25 radial lines that contains $m/n=0.095$ of the points in the Fourier domain. The texture image that represents the noise is created as before. The Frobenius norm of the noise image is scaled to 0.2 of the Frobenius norm of the cartoon image. SACoSaMP uses $k=500$ and $\ell=127064$ ($\ell$ is slightly lower than the number of zeros in the analysis representation of the cartoon image). The PSNR of the reconstructed clean image is 35.07 dB, and the PSNR of the reconstructed noisy image is 23.94 dB. The reason for the relatively low PSNR of the latter is the small number of high frequency points in the binary mask. Regarding the algorithm with split LS step (modified SACoSaMP), the best results are received with $\ell=126514$. The PSNR of the reconstructed clean image is 30.71 dB, and the PSNR of the reconstructed noisy image is 23.29 dB. The number of iterations is 17 comparing to 8 when the LS step is unified. As before, the results of ACoSaMP are not satisfying (we show only its result for $\ell=126514$). The inputs and the various results are shown in Figs. \ref{fig:phantom_clean} - \ref{fig:phantom_acosamp}.


\begin{figure}
 \centering
  \subcaptionbox{{\em Phantom}  \label{fig:phantom_clean}}{%
  \includegraphics[width=0.24\columnwidth]{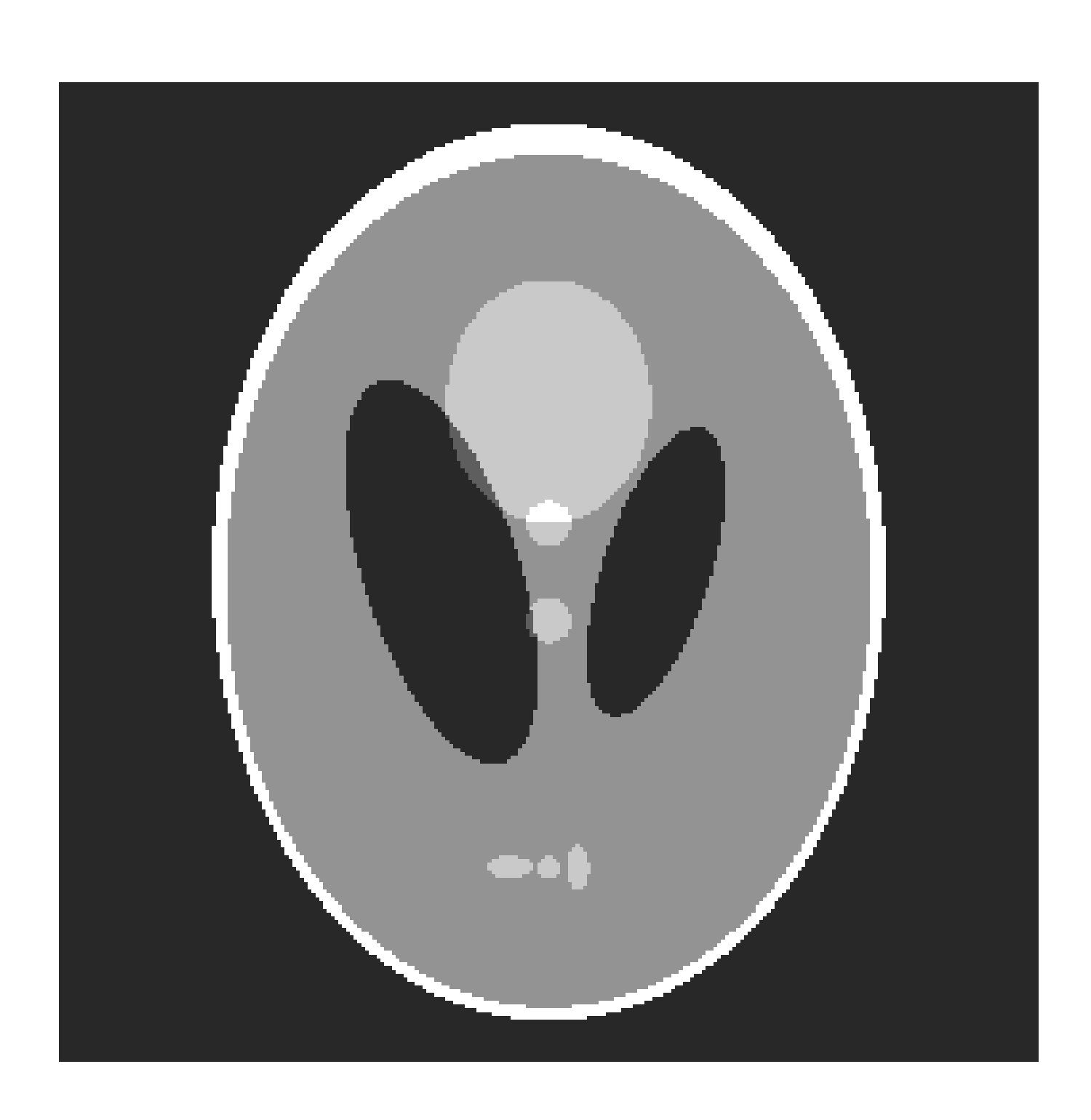}}
  \subcaptionbox{{\em Phantom} + textured noise  \label{fig:phantom_mixed}}{%
  \includegraphics[width=0.24\columnwidth]{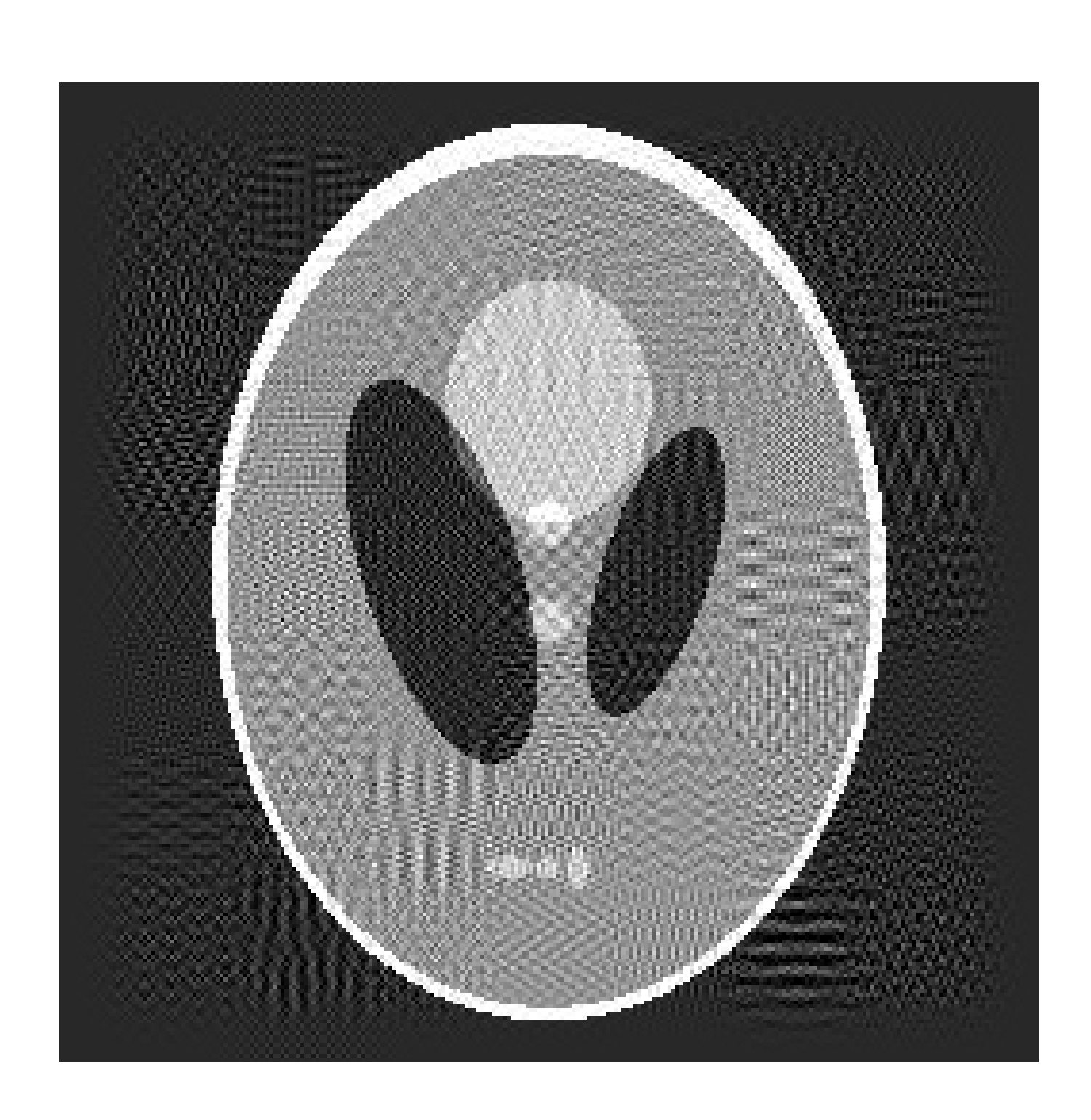}}
  \subcaptionbox{Sampling mask  \label{fig:mask_vardens_256}}{%
  \includegraphics[width=0.24\columnwidth]{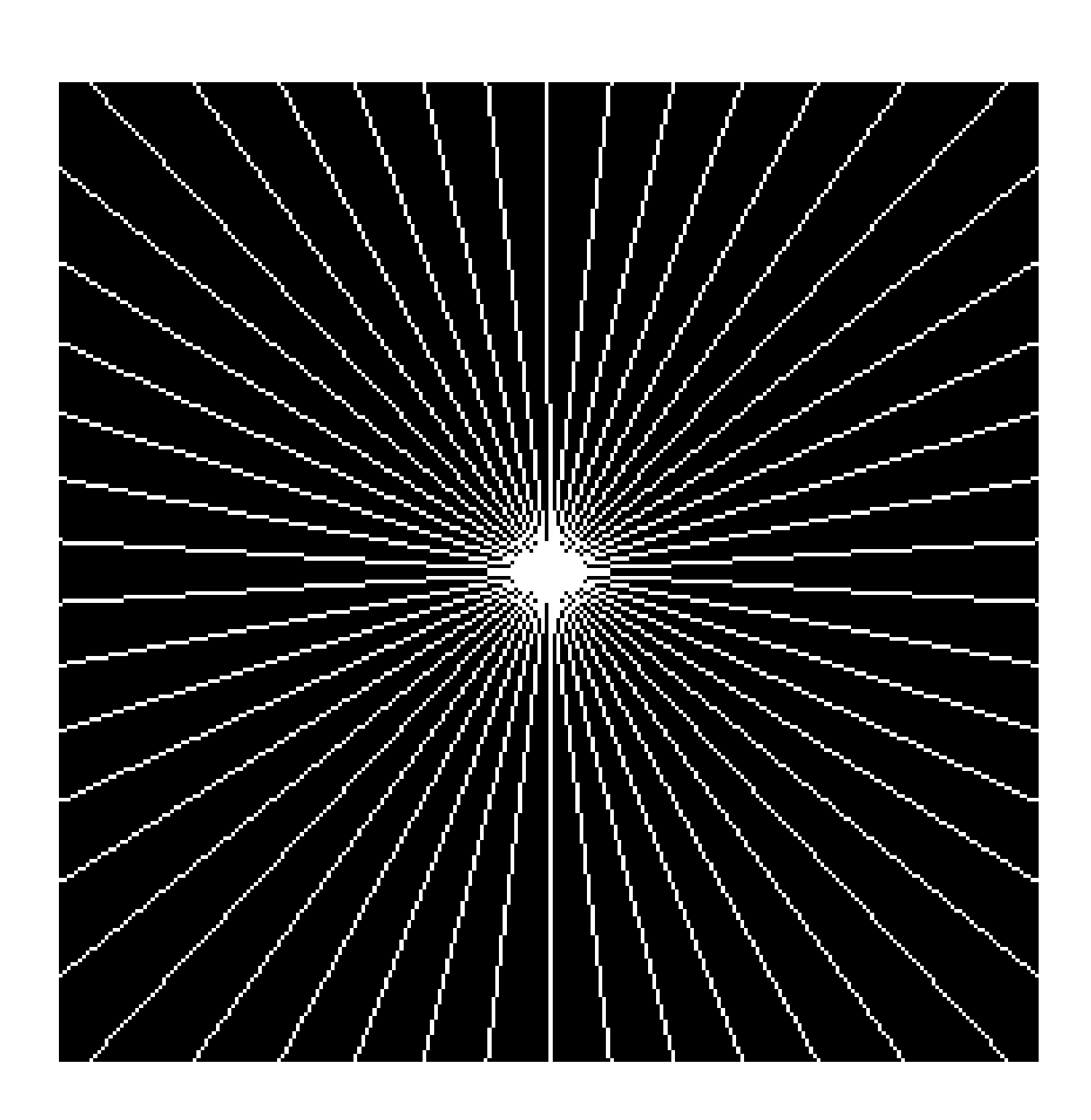}}
  \subcaptionbox{Na\"{\i}ve reconstruction  \label{fig:phantom_trivial}}{%
  \includegraphics[width=0.24\columnwidth]{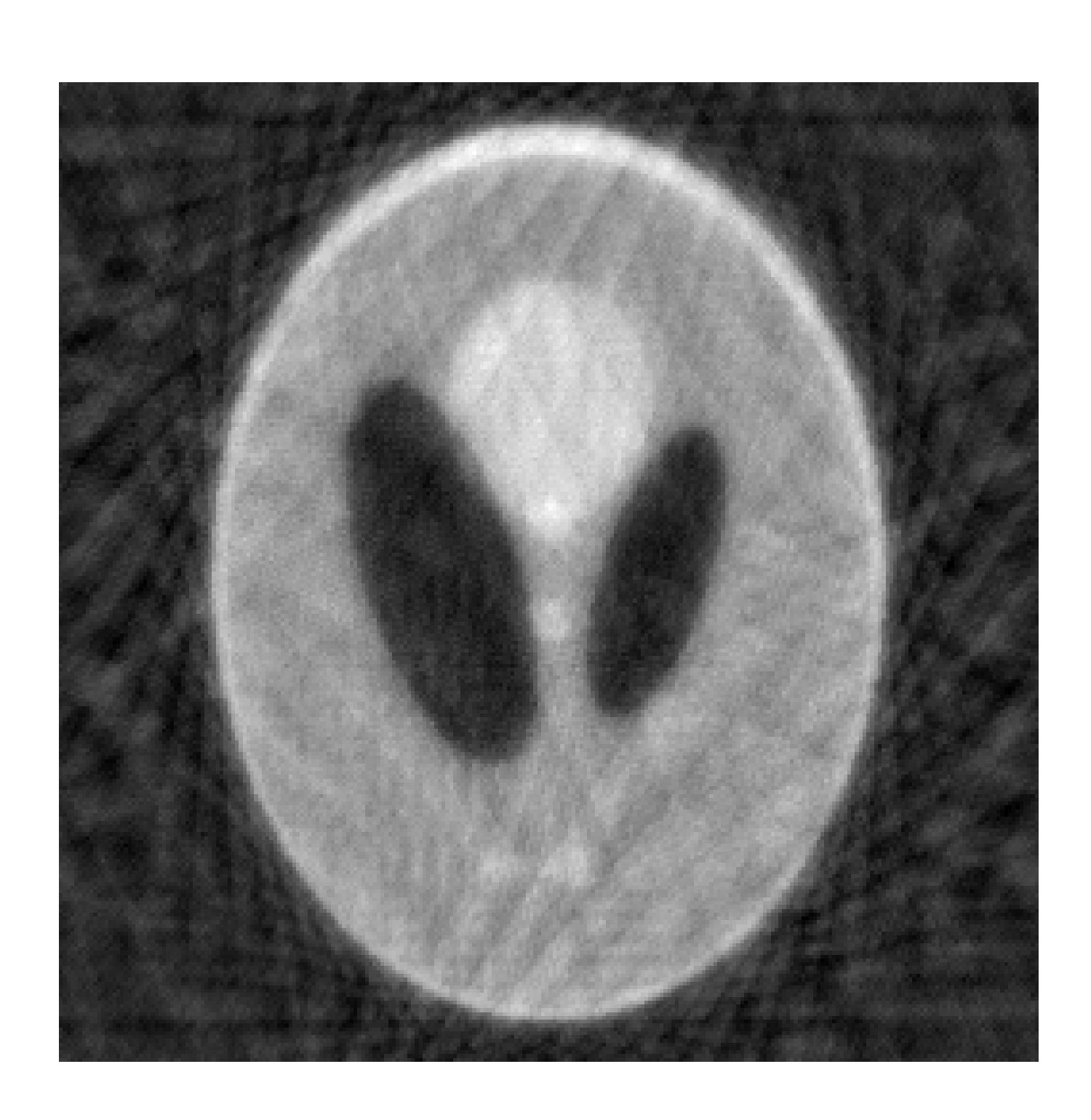}}
  \subcaptionbox{SACoSaMP - {\em phantom} reconstruction  \label{fig:phantom_clean_rec}}{%
  \includegraphics[width=0.24\columnwidth]{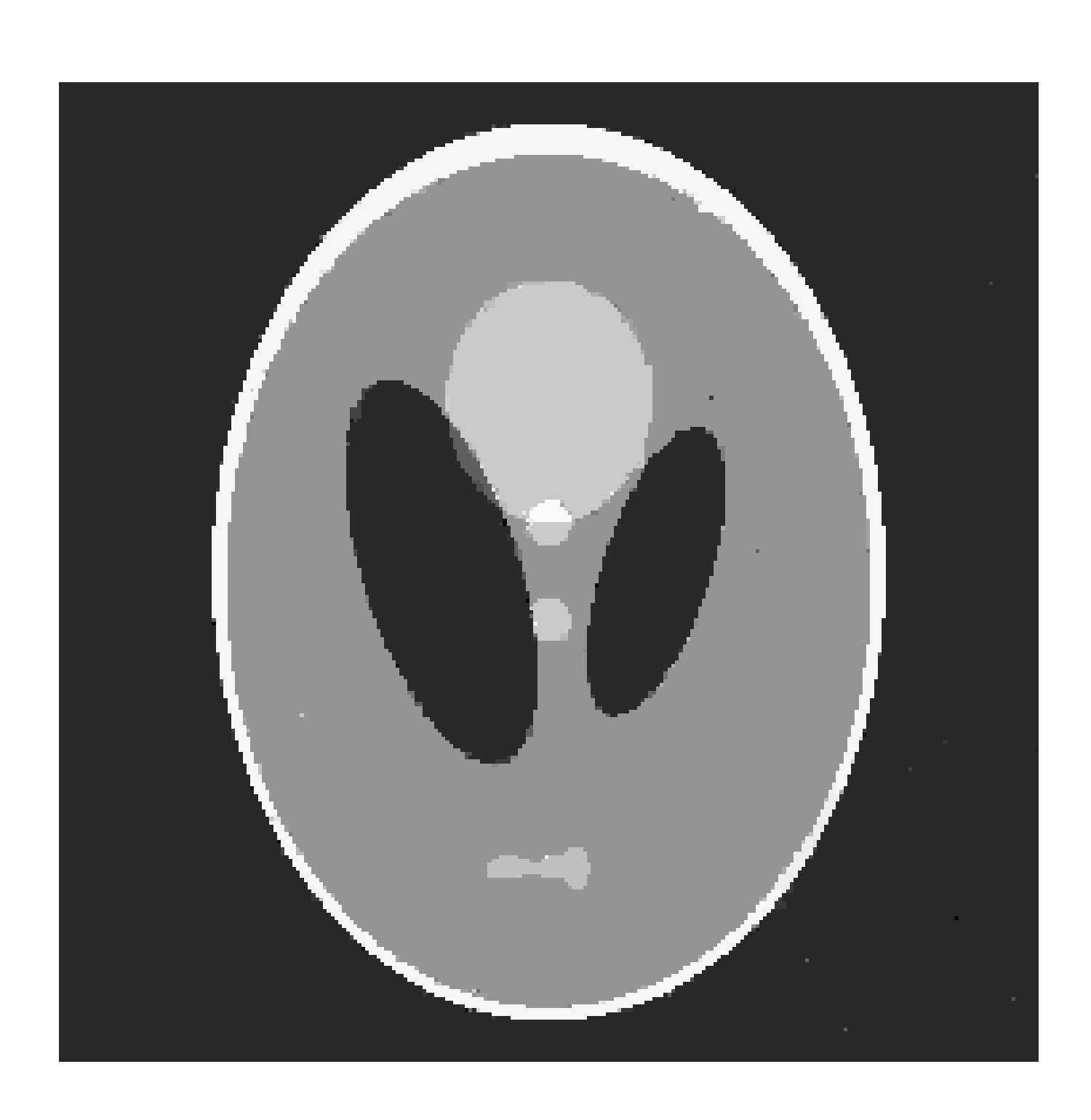}}
  \subcaptionbox{SACoSaMP - {\em phantom} + textured noise reconstruction  \label{fig:phantom_mixed_rec}}{%
  \includegraphics[width=0.24\columnwidth]{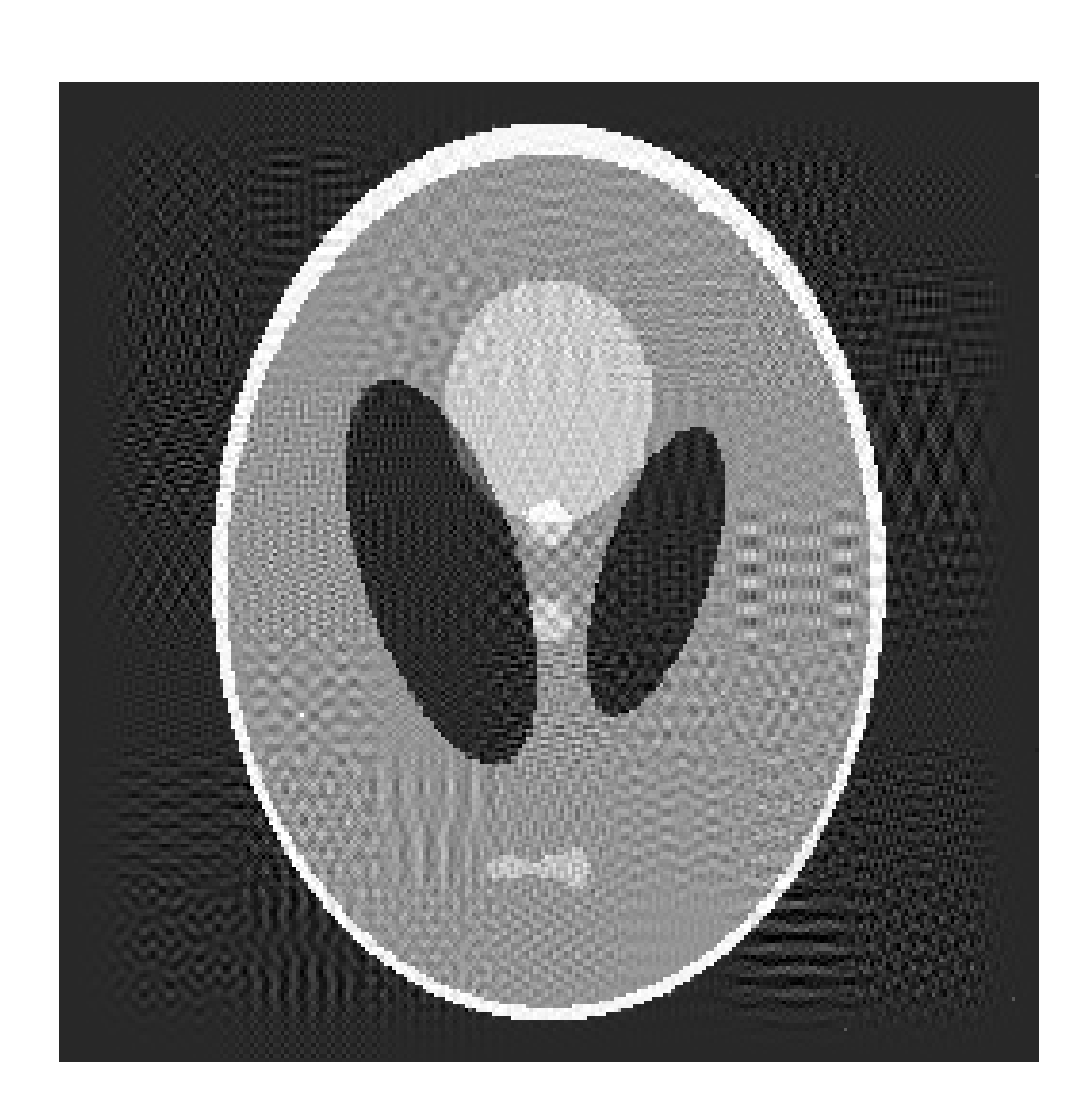}}
  \subcaptionbox{Modified SACoSaMP - {\em phantom} reconstruction  \label{fig:phantom_clean_rec_sep}}{%
  \includegraphics[width=0.24\columnwidth]{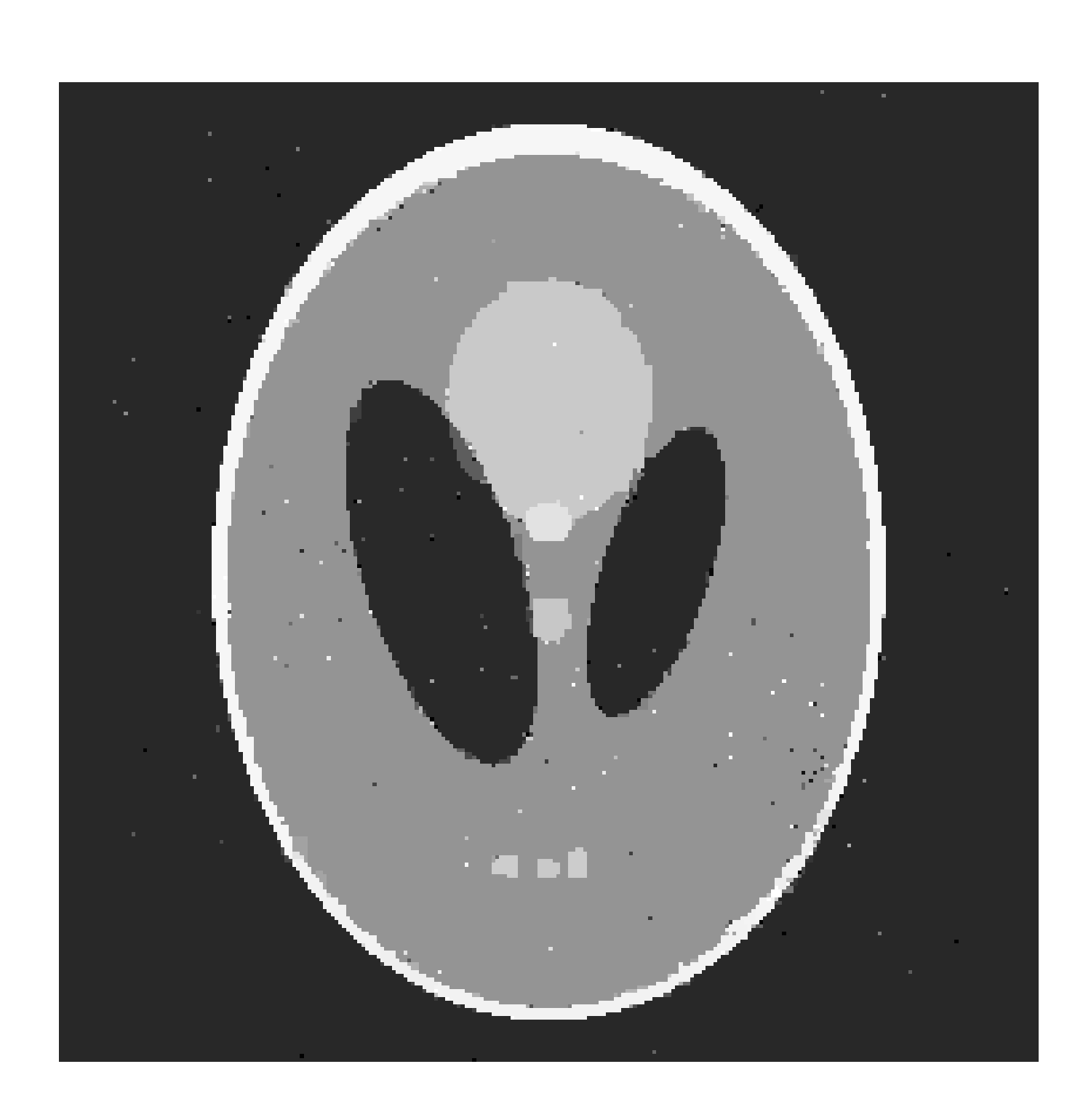}}
  \subcaptionbox{ACoSaMP reconstruction  \label{fig:phantom_acosamp}}{%
  \includegraphics[width=0.24\columnwidth]{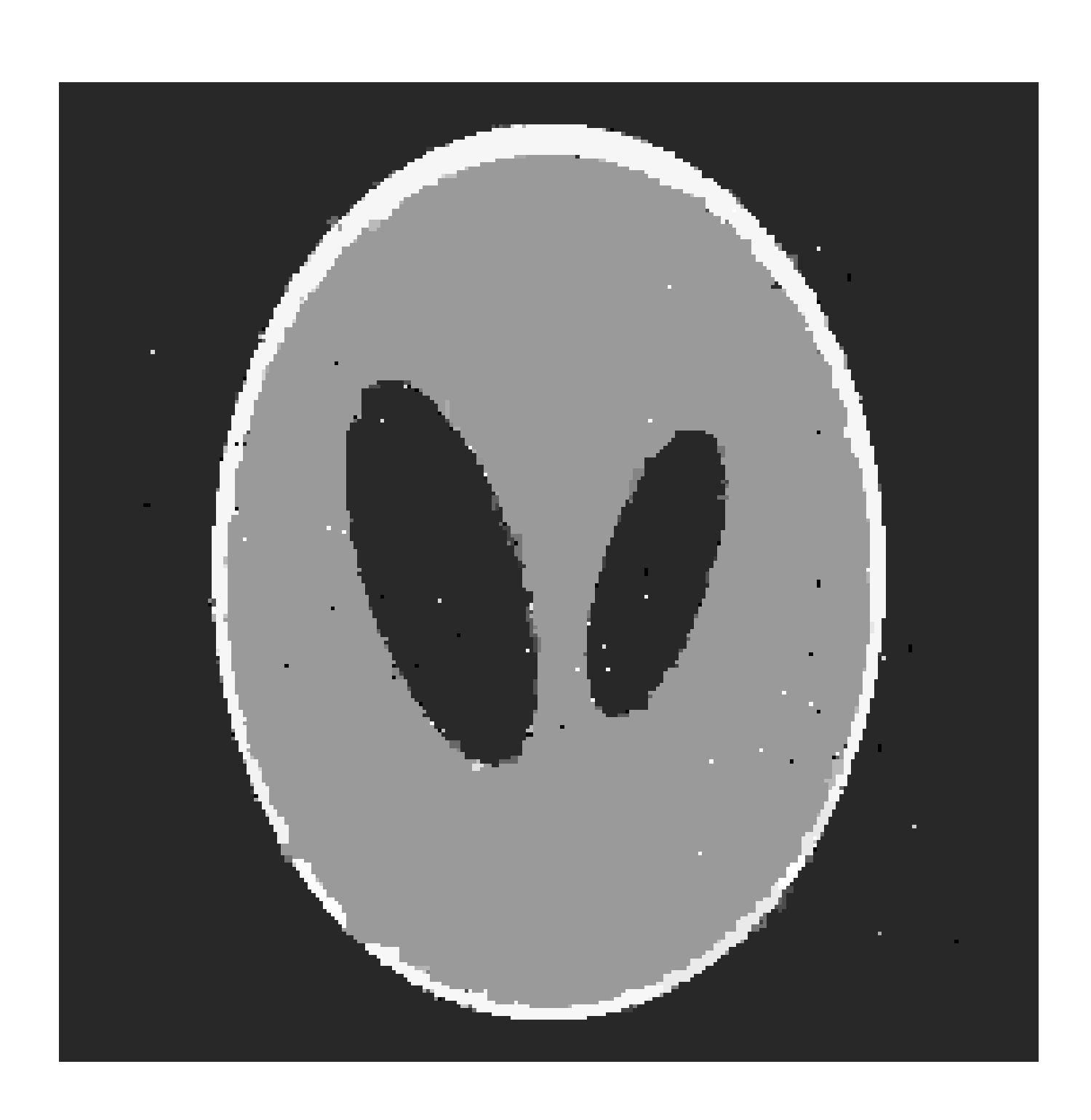}}
  \caption{Experiments inputs and results for modified Shepp-Logan phantom test image. From left to right and top to bottom: {\em phantom} image, noisy {\em phantom} image (textured noise), binary mask for Fourier domain sampling, na\"{\i}ve reconstruction using zero padding and inverse Fourier transform, SACoSaMP reconstruction of {\em phantom} image (cosparse analysis part), SACoSaMP reconstruction of noisy {\em phantom} image (cosparse analysis + sparse synthesis parts), modified SACoSaMP (split LS) reconstruction of {\em phantom} image (cosparse analysis part), and ACoSaMP reconstruction. }\label{fig2}
\end{figure}



\section{Conclusions}
\label{Sec6}

Under the assumption of Gaussian measurement matrix, we studied a generalized version of the CoSaMP algorithm (GCoSaMP), and provided general recovery guarantees that depend on the Gaussian mean width of the union of subspaces in which the signal resides. We discussed the evaluation of the Gaussian mean width and necessary relaxations of the generalized algorithm in various specific models. We proved that when no relaxation is done, the effect of any (non-adversarial) stationary measurement noise vanishes as the number of measurements grows. This is the case in CoSaMP and ADMiRA, though the proof techniques used in previous works did not lead to this general conclusion.
We proposed a new method, inspired by GCoSaMP, for signal recovery in a combined sparse synthesis and cosparse analysis model, and performed experiments that demonstrate its usefulness.

One direction for future research is evaluating the Gaussian mean width of unions of subspaces in specific models for which this quantity is still unknown. Another direction is designing recovery algorithms for combined models, which follows GCoSaMP more closely than our newly proposed SACoSaMP or other solutions such as SpaRCS. Future research may also focus on extending the recovery guarantees of GCoSaMP to measurement matrix with non-Gaussian distribution. This may be done using the proof technique proposed in \cite{Ai14One}.


\appendices 

\section{Proof of Lemma \ref{lemma2}}

Using the triangle inequality, we have
\begin{align}
\label{Eq_lemma2_triangle}
\| \x^t - \x \|_2 & = \| \x^t - \tilde{\x}^t + \tilde{\x}^t - \x \|_2 \\ \nonumber
& \leq \| \tilde{\x}^t - \x^t \|_2 + \| \tilde{\x}^t - \x \|_2
\end{align}
Note that $\x \in  \mathcal{V}_0 \in \mathcal{S}$, and also $\x^t \in \mathcal{V}^t \in \mathcal{S}$. By construction $\x^t$ is the nearest to $\tilde{\x}^t$, in the $\ell_2$-norm sense, among all vectors that reside in one of the subspaces in the set $\mathcal{S}$. Therefore, $\| \tilde{\x}^t - \x^t \|_2 \leq \| \tilde{\x}^t - \x \|_2$, which together with (\ref{Eq_lemma2_triangle}) leads to (\ref{Eq_lemma2}).

\section{Proof of Lemma \ref{lemma1}}

 To simplify the notations we define $\z \triangleq \tilde{\x}^t - \x$. Since $\A\tilde{\x}^t$ is the orthogonal projection of $\y$ onto the subspace defined by $\{\A\u : \u \in \tilde{\mathcal{V}}^t \}$, then for any $\u \in \tilde{\mathcal{V}}^t$ we have
\begin{align}
\label{Eq_lemma1_ortho1}
\langle \A\tilde{\x}^t -\y , \A \u \rangle = 0.
\end{align}
Substituting (\ref{Eq_model}) with simple arithmetics gives
\begin{align}
\label{Eq_lemma1_ortho2}
\langle \z , \A^*\A \u \rangle = \langle \e , \A \u \rangle.
\end{align}
Letting $\mu_1$ and $\eta$ be positive parameters satisfying $\mu_1 \leq \left ( b_m+w(\mathcal U^{4} \cap \Bbb S^{n-1})+\eta \right )^{-2}$ and turning to look at $\left \| \P_{\tilde{\mathcal{V}}^t}\z \right \|_2^2$, we get
\begin{align}
\label{Eq_lemma1_Pt_square}
\left \| \P_{\tilde{\mathcal{V}}^t}\z \right \|_2^2  &= \langle \z , \P_{\tilde{\mathcal{V}}^t}\z \rangle \\ \nonumber
&= \langle \z , (\I_n - \mu_1 \A^*\A) \P_{\tilde{\mathcal{V}}^t} \z \rangle + \langle \z , \mu_1 \A^*\A \P_{\tilde{\mathcal{V}}^t} \z \rangle \\ \nonumber
&= \langle \z , (\I_n - \mu_1 \A^*\A) \P_{\tilde{\mathcal{V}}^t} \z \rangle + \langle \e , \mu_1 \A \P_{\tilde{\mathcal{V}}^t} \z \rangle \\ \nonumber
&= \langle \z , \P_{\tilde{\mathcal{V}}^t + \mathcal{V}_0} (\I_n - \mu_1 \A^*\A) \P_{\tilde{\mathcal{V}}^t + \mathcal{V}_0} \P_{\tilde{\mathcal{V}}^t} \z \rangle + \langle \e , \mu_1 \A \P_{\tilde{\mathcal{V}}^t} \z \rangle \\ \nonumber
& \leq \| \z \|_2 \left \| \P_{\tilde{\mathcal{V}}^t + \mathcal{V}_0} (\I_n - \mu_1 \A^*\A) \P_{\tilde{\mathcal{V}}^t + \mathcal{V}_0} \right \| \left \| \P_{\tilde{\mathcal{V}}^t} \z \right \|_2 + \mu_1 \langle \e , \A \P_{\tilde{\mathcal{V}}^t} \z \rangle.
\end{align}
The third equality follows from (\ref{Eq_lemma1_ortho2}) with $\u=\P_{\tilde{\mathcal{V}}^t}\z$, the fourth equality follows from $\P_{\tilde{\mathcal{V}}^t + \mathcal{V}_0}\z=\z$ and $\P_{\tilde{\mathcal{V}}^t} =\P_{\tilde{\mathcal{V}}^t + \mathcal{V}_0}\P_{\tilde{\mathcal{V}}^t}$, and the inequality follows from Cauchy-Schwarz inequality.
Define $\tilde{\z} \triangleq \frac{\P_{\tilde{\mathcal{V}}^t}\z}{\left \| \P_{\tilde{\mathcal{V}}^t} \z \right \|_2}$ and $\tilde{\Omega} \triangleq \tilde{\mathcal{V}}^t \cap \Bbb S^{n-1}$, thus $\tilde{\z} \in \tilde{\Omega}$. Note that $\tilde{\Omega}$ may differ from iteration to iteration but it is always a subset of $\mathcal U^{3} \cap \Bbb S^{n-1}$. Dividing (\ref{Eq_lemma1_Pt_square}) by $\left \| \P_{\tilde{\mathcal{V}}^t} \z \right \|_2$ we have
\begin{align}
\label{Eq_lemma1_Pt}
\left \| \P_{\tilde{\mathcal{V}}^t} \z \right \|_2  \leq \| \z \|_2 \| \P_{\tilde{\mathcal{V}}^t + \mathcal{V}_0} (\I_n - \mu_1 \A^*\A) \P_{\tilde{\mathcal{V}}^t + \mathcal{V}_0} \| + \mu_1 \|\e\|_2 \langle \frac{\e}{\|\e\|_2} , \A \tilde{\z} \rangle.
\end{align}
Focusing on the last term, for a fixed error vector $\e$, note that $\g \triangleq \A^* \frac{\e}{\|\e\|_2} \sim \mathcal{N}(\0,\I_n)$. Applying standard results on concentration of supremum of Gaussian processes we get that
\begin{align}
\label{Eq_lemma1_e_term}
\langle \frac{\e}{\|\e\|_2} , \A \tilde{\z} \rangle & \leq \underset{\tilde{\z} \in \tilde{\Omega}}{\operatorname{sup}} \langle  \g , \tilde{\z} \rangle \\ \nonumber
& \leq w(\mathcal U^{3} \cap \Bbb S^{n-1}) + \eta,
\end{align}
holds for all $t$, with probability at least $1-\mathrm{e}^{-\frac{\eta^2}{2}}$.
Using Lemma \ref{lemma_intro}, (\ref{Eq_lemma1_Pt}) and (\ref{Eq_lemma1_e_term}) we have that
\begin{align}
\label{Eq_lemma1_Pt_final}
\left \| \P_{\tilde{\mathcal{V}}^t} \z \right \|_2  \leq \rho_{1}(\eta) \| \z \|_2 + \xi_{1}(\eta) \|\e\|_2,
\end{align}
holds for all t, with probability at least $1-3\mathrm{e}^{-\frac{\eta^2}{2}}$, where $\rho_{1}(\eta)$ and $\xi_{1}(\eta)$ are defined in (\ref{Eq_theorem_rho1}) and (\ref{Eq_theorem_xi1}). Plugging (\ref{Eq_lemma1_Pt_final}) in $\left \| \z \right \|_2^2 = \left \| \Q_{\tilde{\mathcal{V}}^t} \z \right \|_2^2 + \left \| \P_{\tilde{\mathcal{V}}^t} \z \right \|_2^2$  gives
\begin{align}
\label{Eq_lemma1_x_poly}
(1- \rho_{1}^2(\eta)) \left \| \z \right \|_2^2 - 2  \rho_{1}(\eta)  \xi_{1}(\eta) \|\e\|_2  \left \| \z \right \|_2 -  \left \| \Q_{\tilde{\mathcal{V}}^t} \z \right \|_2^2 - \xi_{1}^2(\eta)  \left \|\e \right \|_2^2 \leq 0. 
\end{align}
The left-hand side of the last equation is a second order polynomial of $\left \| \z \right \|_2$ with positive leading coefficient, hence it follows that $\left \| \z \right \|_2$ is upper bounded by the larger root. Equation (\ref{Eq_lemma1}) is obtained by calculating this bound and using the fact that $\sqrt{a+b} \leq \sqrt{a} + \sqrt{b}$ for nonnegative $a$ and $b$.

\section{Proof of Lemma \ref{lemma3}}
\label{AppendixC}

Our proof is inspired by the proof of Lemma A.1 in \cite{davenport2013signal}. The main differences between the proofs are in the later steps, where our arguments follow from concentration of measure, while those in \cite{davenport2013signal} follow from a generalization of the RIP.
We start with upper bounding  the term $\| \Q_{\mathcal{V}_{\Delta}^t}(\x - \x^{t-1}) \|_2$. To simplify the notations we define $\v \triangleq \x - \x^{t-1}$, $\overline{\mathcal{V}} \triangleq \mathcal{V}_0 + \mathcal{V}^{t-1}$, and note that $\tilde{\v}=\A^*(\y-\A\x^{t-1})=\A^*\A\v+\A^*\e$. Letting $\mu_2$ and $\eta$ be positive parameters satisfying $\mu_2 \leq \left ( b_m+w(\mathcal U^{4} \cap \Bbb S^{n-1})+\eta \right )^{-2}$, we get
\begin{align}
\label{Eq_lemma3_main}
\| \Q_{\mathcal{V}_{\Delta}^t}\v \|_2 & = \| \v -  \P_{\mathcal{V}_{\Delta}^t}\v \|_2 \\ \nonumber
& \leq \| \v -  \P_{\mathcal{V}_{\Delta}^t}(\mu_2 \tilde{\v}) \|_2 \\ \nonumber
& = \| \v - \mu_2 \P_{\overline{\mathcal{V}} + \mathcal{V}_{\Delta}^t}\tilde{\v} + \mu_2 \P_{\overline{\mathcal{V}} + \mathcal{V}_{\Delta}^t}\tilde{\v} - \mu_2 \P_{\mathcal{V}_{\Delta}^t} \tilde{\v} \|_2  \\ \nonumber
& \leq \| \v - \mu_2 \P_{\overline{\mathcal{V}} + \mathcal{V}_{\Delta}^t}\tilde{\v} \|_2 + \mu_ 2 \| \P_{\overline{\mathcal{V}} + \mathcal{V}_{\Delta}^t}\tilde{\v} - \P_{\mathcal{V}_{\Delta}^t} \tilde{\v} \|_2  \\ \nonumber
& \leq \| \v - \mu_2 \P_{\overline{\mathcal{V}} + \mathcal{V}_{\Delta}^t}\A^*\A\v \|_2 + \mu_2 \| \P_{\overline{\mathcal{V}} + \mathcal{V}_{\Delta}^t}\A^*\e \|_2 + \mu_ 2 \| \P_{\overline{\mathcal{V}} + \mathcal{V}_{\Delta}^t}\tilde{\v} - \P_{\mathcal{V}_{\Delta}^t} \tilde{\v} \|_2.
\end{align}
The first inequality follows from the fact that $\P_{\mathcal{V}_{\Delta}^t}\v$ is the nearest to $\v$,  in the $\ell_2$-norm sense, among all vectors in $\mathcal{V}_{\Delta}^t$. The last two inequalities follow from the triangle inequality. Focusing on the last term
\begin{align}
\label{Eq_lemma3_lastTerm}
\mu_ 2 \| \P_{\overline{\mathcal{V}} + \mathcal{V}_{\Delta}^t}\tilde{\v} - \P_{\mathcal{V}_{\Delta}^t} \tilde{\v} \|_2 & \leq \mu_ 2 \| \P_{\overline{\mathcal{V}} + \mathcal{V}_{\Delta}^t}\tilde{\v} - \P_{\overline{\mathcal{V}}} \tilde{\v} \|_2 \\ \nonumber
& = \mu_ 2 \| \P_{\overline{\mathcal{V}} + \mathcal{V}_{\Delta}^t}\tilde{\v} - \P_{\overline{\mathcal{V}}} \P_{\overline{\mathcal{V}} + \mathcal{V}_{\Delta}^t} \tilde{\v} \|_2 \\ \nonumber
& \leq \mu_ 2 \| \P_{\overline{\mathcal{V}} + \mathcal{V}_{\Delta}^t}\tilde{\v} - \P_{\overline{\mathcal{V}}} (\frac{1}{\mu_2}\v + \A^*\e) \|_2 \\ \nonumber
& = \| \mu_ 2 \P_{\overline{\mathcal{V}} + \mathcal{V}_{\Delta}^t} (\A^*\A\v + \A^*\e) - \v - \mu_ 2 \P_{\overline{\mathcal{V}}} \A^*\e \|_2 \\ \nonumber
& \leq \| \v - \mu_ 2 \P_{\overline{\mathcal{V}} + \mathcal{V}_{\Delta}^t} \A^*\A\v \|_2 + \mu_ 2 \| \P_{\overline{\mathcal{V}} + \mathcal{V}_{\Delta}^t} \A^*\e - \P_{\overline{\mathcal{V}}} \A^*\e \|_2.
\end{align}
The first inequality follows from the fact that both $\mathcal{V}_{\Delta}^t$ and $\overline{\mathcal{V}}$ are subsets of $\mathcal{V}_{\Delta}^t + \overline{\mathcal{V}}$, and that by construction $\P_{\mathcal{V}_{\Delta}^t}\tilde{\v}$ is closer to $\tilde{\v}$ than $\P_{\overline{\mathcal{V}}}\tilde{\v}$ (since $\mathcal{V}_{\Delta}^t, \overline{\mathcal{V}} \in \mathcal{S}^2$). For more details look at the discussion above (13) in \cite{davenport2013signal}. The second inequality results from $\P_{\overline{\mathcal{V}}} \P_{\overline{\mathcal{V}} + \mathcal{V}_{\Delta}^t} \tilde{\v}$ being the nearest to $\P_{\overline{\mathcal{V}} + \mathcal{V}_{\Delta}^t} \tilde{\v}$ in the $\ell_2$-norm sense among all vectors in $\overline{\mathcal{V}}$, and the last inequality uses the triangle inequality. Note also that
\begin{align}
\label{Eq_lemma3_lastTerm2}
\| \P_{\overline{\mathcal{V}} + \mathcal{V}_{\Delta}^t} \A^*\e - \P_{\overline{\mathcal{V}}} \A^*\e \|_2 & = \| \P_{\overline{\mathcal{V}} + \mathcal{V}_{\Delta}^t} \A^*\e - \P_{\overline{\mathcal{V}}} \P_{\overline{\mathcal{V}} + \mathcal{V}_{\Delta}^t} \A^*\e \|_2 \\ \nonumber
& = \| (\I_n - \P_{\overline{\mathcal{V}}}) \P_{\overline{\mathcal{V}} + \mathcal{V}_{\Delta}^t} \A^*\e \|_2 \\ \nonumber
& \leq \| \P_{\overline{\mathcal{V}} + \mathcal{V}_{\Delta}^t} \A^*\e \|_2,
\end{align}
where the last inequality is due to fact that $\I_n - \P_{\overline{\mathcal{V}}}$ is an orthogonal projection and hence its spectral norm is bounded by 1.
Using (\ref{Eq_lemma3_main}), (\ref{Eq_lemma3_lastTerm}) and (\ref{Eq_lemma3_lastTerm2}), we get
\begin{align}
\label{Eq_lemma3_main2}
\| \Q_{\mathcal{V}_{\Delta}^t}\v \|_2 \leq 2 \| \v - \mu_2 \P_{\overline{\mathcal{V}} + \mathcal{V}_{\Delta}^t}\A^*\A\v \|_2 + 2 \mu_2 \| \P_{\overline{\mathcal{V}} + \mathcal{V}_{\Delta}^t}\A^*\e \|_2.
\end{align}
Dealing with the first term of (\ref{Eq_lemma3_main2}), using $\P_{\overline{\mathcal{V}} + \mathcal{V}_{\Delta}^t} \v=\v$ and Lemma \ref{lemma_intro}, we get 
\begin{align}
\label{Eq_lemma3_main2_1}
\| \v - \mu_2 \P_{\overline{\mathcal{V}} + \mathcal{V}_{\Delta}^t}\A^*\A\v \|_2 & = \| \P_{\overline{\mathcal{V}} + \mathcal{V}_{\Delta}^t} (\I_n - \mu_2 \A^*\A ) \P_{\overline{\mathcal{V}} + \mathcal{V}_{\Delta}^t} \v \|_2 \\ \nonumber
& \leq \left \| \P_{\overline{\mathcal{V}} + \mathcal{V}_{\Delta}^t} (\I_n - \mu_2 \A^*\A ) \P_{\overline{\mathcal{V}} + \mathcal{V}_{\Delta}^t} \right \| \| \v \|_2 \\ \nonumber
& \leq \rho_{2}(\eta) \| \v \|_2,
\end{align}
where the last inequality holds for all $t$, with probability at least $1-2\mathrm{e}^{-\frac{\eta^2}{2}}$, and $\rho_{2}(\eta)$ is defined in (\ref{Eq_theorem_rho2}).
Focusing on the second term in (\ref{Eq_lemma3_main2}), using the fact that $\g \triangleq \A^* \frac{\e}{\|\e\|_2} \sim \mathcal{N}(\0,\I_n)$ for a fixed error vector, and the fact that if $\z \in \Bbb R^n$ has unit Euclidean norm, then $\P_{\overline{\mathcal{V}} + \mathcal{V}_{\Delta}^t} \z \in \mathcal U^{4} \cap \Bbb B^n$, we get
\begin{align}
\label{Eq_lemma3_main2_2}
\| \P_{\overline{\mathcal{V}} + \mathcal{V}_{\Delta}^t}\A^*\e \| & = \underset{\z : \|\z\|_2=1}{\operatorname{sup}} \langle \z , \P_{\overline{\mathcal{V}} + \mathcal{V}_{\Delta}^t}\A^*\e \rangle \\ \nonumber
& = \underset{\z : \|\z\|_2=1}{\operatorname{sup}} \langle \P_{\overline{\mathcal{V}} + \mathcal{V}_{\Delta}^t} \z , \g \rangle \|\e\|_2 \\ \nonumber
& \leq \left [ w(\mathcal U^{4} \cap \Bbb B^n) + \eta \right ] \|\e\|_2 \\ \nonumber
& = \left [ w(\mathcal U^{4} \cap \Bbb S^{n-1}) + \eta  \right ] \|\e\|_2.
\end{align}
The inequality follows from standard results on concentration of supremum of Gaussian processes, similarly to (\ref{Eq_lemma1_e_term}), and holds for all $t$, with probability at least $1-\mathrm{e}^{-\frac{\eta^2}{2}}$.
The last equality follows from the fact that for a linear subspace $\mathcal L$ we have $\underset{\z \in \mathcal{L} \cap \Bbb B^n}{\operatorname{sup}} \langle \g,\z \rangle = \underset{\z \in \mathcal{L} \cap \Bbb S^{n-1}}{\operatorname{sup}} \langle \g,\z \rangle$. In details, if $\g$ is orthogonal to $\mathcal{L}$ then this is trivial, else, the maximum is achieved at an extreme point of $\mathcal{L} \cap \Bbb B^n$, i.e., a point in $\mathcal{L} \cap \Bbb S^{n-1}$. Similarly, for a union of linear subspaces the maximum is achieved at an extreme point in the subspace that is closest to $\g$ in angular distance.
We finish with the observation
\begin{align}
\label{Eq_lemma3_main3}
\| \Q_{\tilde{\mathcal{V}}^t}(\tilde{\x}^t-\x) \|_2 & \leq \| \Q_{\tilde{\mathcal{V}}^t}(\tilde{\x}^t-\x^{t-1}) \|_2 + \| \Q_{\tilde{\mathcal{V}}^t}(\x - \x^{t-1}) \|_2 \\ \nonumber
& = \| \Q_{\tilde{\mathcal{V}}^t}(\x - \x^{t-1}) \|_2 \\ \nonumber
& \leq \| \Q_{\mathcal{V}_{\Delta}^t}(\x - \x^{t-1}) \|_2.
\end{align}
The first inequality uses the triangle inequality, the equality follows from $\tilde{\x}^t, \x^{t-1} \in \tilde{\mathcal{V}}^t$, and the last inequality follows from $\mathcal{V}_{\Delta}^t \subset \tilde{\mathcal{V}}^t$.
Combining (\ref{Eq_lemma3_main2}) - (\ref{Eq_lemma3_main3}), we get (\ref{Eq_lemma3}).


\bibliographystyle{ieeetr}
\bibliography{paper_ver_for_arXiv}

\end{document}